\documentclass[12pt]{amsart}
\usepackage[all, cmtip]{xy}

\usepackage[colorlinks=true,
            linkcolor=blue,
            urlcolor=red,
            citecolor=magenta]{hyperref}

\usepackage{times,amsfonts,amsmath,amstext,amsbsy,amssymb,
amsopn,amsthm,upref,eucal, mathrsfs, xcolor}

\usepackage{url}

\newtheorem{defn}{Definition}[section]
\newtheorem{thm}[defn]{Theorem}
\newtheorem{lem}[defn]{Lemma}

\newtheorem{cor}[defn]{Corollary}
\newtheorem{prop}[defn]{Proposition}
\newtheorem{rem}[defn]{Remark}

\newcommand{\N}{\mathbb{N}}
\newcommand{\R}{\mathbb{R}}
\newcommand{\C}{\mathbb{C}}
\newcommand{\Z}{\mathbb{Z}}
\newcommand{\Q}{\mathbb{Q}}

\newcommand \tr {{\mathrm{tr}}}

\newcommand \Conf {{\mathrm {Conf}}}

\input xy
\xyoption{all}

\begin{document}
\title[The Explicit formulae for the scaling limits]{The Explicit formulae for  scaling limits in the
ergodic decomposition of infinite Pickrell measures  }

\begin{abstract}
{The main result of this paper, Theorem \ref{main-result}, gives explicit formulae for the kernels  of the ergodic decomposition measures  for infinite Pickrell measures on spaces of infinite complex matrices. The kernels are obtained as the scaling limits of Christoffel-Uvarov deformations of Jacobi
orthogonal polynomial ensembles.
}\end{abstract}

\author{ Alexander I. Bufetov
,  Yanqi Qiu }

\address{A.B.:Institut de Math\'ematiques de Marseille, Aix-Marseille Universit{\'e}, CNRS, Marseille}
\address{Steklov Institute of Mathematics,
Moscow}
\address{Institute for Information Transmission Problems,
 Moscow}
\address{National Research University Higher School of Economics,
 Moscow}
\address{Rice University, Houston TX}

\address{Y. Q.: Institut de Math\'ematiques de Marseille,  Aix-Marseille Universit{\'e}, Marseille}

\maketitle

\section{Introduction.}
\subsection{Outline of the main results}
\subsubsection{Pickrell measures}
We start by recalling the definition of Pickrell measures \cite{Pickrell90}.
Our presentation follows \cite{Bufetov_inf-deter}.

Given a parameter $s \in \R$ and a natural number $n$,  consider 
a measure $\mu_n^{(s)}$ on the space $\text{Mat}(n, \C)$ of $n\times n$-complex matrices, given by the formula
\begin{equation}\label{pick-def}
{\mu}_n^{(s)}=\mathrm{const}_{n,s}\det(1+{z}^*{z})^{-2n-s}dz.
\end{equation}
Here $dz$ is the Lebesgue measure on the space of  matrices, and $\mathrm{const}_{n,s}$ a normalization constant whose choice will be explained later. Note that the measure $\mu_n^{(s)}$ is finite if $s>-1$ and infinite if $s\leq -1$.

If the constants $\mathrm{const}_{n,s}$ are chosen appropriately, then  the sequence of measures  (\ref{pick-def}) has
the Kolmogorov property of consistency  under natural projections:  the push-forward of the measure
${\mu}_{n+1}^{(s)}$ under the natural projection of cutting the $n\times n$-corner of a $(n+1)\times (n+1)$-matrix is precisely the  measure ${\mu}_n^{(s)}$. This consistency property  also holds for infinite  
measures provided $n$ is sufficiently large. The consistency property and the Kolmogorov Existence Theorem allows one  to define the Pickrell measure $\mu^{(s)}$ on the space of infinite complex matrices $\text{Mat}(\N, \C)$, which   is finite if $s>-1$ and infinite if $s\leq -1$.

Let $U(\infty)$ be the infinite unitary group $$U(\infty) = \bigcup_{n\in\N} U(n),$$ and let $G = U(\infty) \times U(\infty)$. Groups like $U(\infty)$ or $G$ are considered as nice ``big groups'', they are non-locally compact groups, but are the inductive limits of compact ones.  

The space $\text{Mat}(\N, \C)$ can be naturally considered as a $G$-space given by the action \begin{align*} T_{u_1, u_2}( z) = u_1 z u_2^*, \text{ for } (u_1, u_2) \in G, z \in \text{Mat}(\N, \C).\end{align*} By definition, the Pickrell measures are $G$-invariant. The ergodic decomposition of Pickrell measures with respect to $G$-action was studied in \cite{BO-infinite-matrices} in finite case and \cite{Bufetov_inf-deter} in infinite case. The ergodic $G$-invariant probability measures on $\text{Mat}(\N, \C)$ admit an explicit classification
due to Pickrell \cite{Pickrell90} and to which Olshanski and Vershik \cite{Olsh-Vershik} gave a different approach: let $\mathfrak{M}_{\text{erg}} (\text{Mat}(\N, \C))$ be the set of ergodic probability measures and define the Pickrell set by \begin{align*} \Omega_P = \left\{ \omega = (\gamma, x) : x = (x_1 \ge x_2 \ge \dots \ge x_i \ge \dots \ge 0), \sum_{i = 1}^\infty x_i  \le \gamma\right\},\end{align*} then there is a natural identification:  $$\begin{array}{ccc} \Omega_P & \leftrightarrow & \mathfrak{M}_{\text{erg}} (\text{Mat}(\N, \C))  \\ \omega & \leftrightarrow &  \eta_\omega\end{array}.$$

Set \begin{align*} \Omega_P^{0} : = \left\{ \omega = (\gamma, x) \in \Omega_P: x_i > 0 \text{ \, for all $i$, and \, } \gamma = \sum_{i = 1}^\infty x_i        \right\}. \end{align*}

The finite Pickrell measures $\mu^{(s)}$ admit the following unique ergodic decomposition \begin{align}\label{erg-dec} \mu^{(s)} = \int_{\Omega_P} \eta_\omega  d \overline{\mu}^{(s)} (\omega). \end{align} 
 Borodin and Olshanski \cite{BO-infinite-matrices} proved that the decomposition measures $\overline{\mu}^{(s)}$ live on $\Omega_P^0$, i.e., $\overline{\mu}^{(s)} (\Omega_P\setminus \Omega_P^0)  = 0 $. Let  $\mathbb{B}^{(s)}$ denote the push-forward of the following map: $$\begin{array}{cccc} \text{conf}: & \Omega_P^0 & \rightarrow & \text{Conf} (( 0,  \infty)) \\  & \omega  & \mapsto & \{ x_1, x_2, \dots, x_i, \dots \} \end{array}. $$ The above $\overline{\mu}^{(s)}$-almost sure bijection  identifies the decomposition measure $\overline{\mu}^{(s)}$ on $\Omega_P$ with the measure $\mathbb{B}^{(s)}$ on $\text{Conf}((0, \infty))$, for this reason, the measure $\mathbb{B}^{(s)}$ will also be called the decomposition measure of the Pickrell measure $\mu^{(s)}$. It is showed that $\mathbb{B}^{(s)}$ is a determinantal measure on $\text{Conf}((0, \infty))$ with correlation kernel \begin{align}\label{bessel-kernel-mod} J^{(s)} (x_1,x_2) =  \frac{1}{x_1x_2} \int_0^1 J_s\left(2\sqrt{\frac{t}{x_1}}\right) J_s\left(2\sqrt{\frac{t}{x_2}}\right)dt .\end{align}  The change of variable $y = 4/x$ reduces the kernel $J^{(s)}$ to the well-known kernel $\widetilde{J}^{(s)}$ of the Bessel point process of Tracy and Widom in \cite{Tracy-Widom94}: \begin{align*} \widetilde{J}^{(s)} (x_1, x_2) = \frac{1}{4} \int_0^1 J_s(\sqrt{tx_1}) J_s(\sqrt{tx_2}) dt .  \end{align*}

When $s \le -1$, the ergodic decomposition of the infinite Pickrell measure $\mu^{(s)}$ was described in \cite{Bufetov_inf-deter}, the decomposition formula takes the same form  as  \eqref{erg-dec}, while this time, the decomposition measure $\overline{\mu}^{(s)}$ is an infinite  measure on $\Omega_P$ and again, we have $\overline{\mu}^{(s)} (\Omega_P \setminus \Omega_P^0) = 0$. The $\overline{\mu}^{(s)}$-almost sure bijection $\omega \rightarrow \text{conf}(\omega)$ identifies $\overline{\mu}^{(s)}$ with an infinite determinantal measure  $\mathbb{B}^{(s)}$ on $\text{Conf} ((0,  \infty))$. One suitable way to describe $\mathbb{B}^{(s)}$ is via the  multiplicative functionals, for which we recall the definition: a multiplicative functional on $\text{Conf}((0,  \infty))$ is obtained by taking the product of the values of a fixed nonnegative function over all particles of a configuration: \begin{align*}  \Psi_g (X) = \prod_{x \in X} g(x) \text{ \, for any  $X \in \text{Conf}((0, \infty))$} . \end{align*} If the function $g: (0, \infty) \rightarrow (0, 1)$ is suitably chosen, then \begin{align} \label{deter-proba} \frac{ \Psi_g \mathbb{B}^{(s)} }{\int_{\textnormal{Conf} ( (0,  \infty) )} \Psi_g d  \mathbb{B}^{(s)}} \end{align} is a determinantal measure on $\text{Conf}((0,  \infty)) $ whose correlation kernel coincides with that of  an orthogonal projection $\Pi^g: L^2(0, \infty) \rightarrow L^2(0, \infty)$. Note that the range $\text{Ran} (\Pi^g)$ of this projection is explicitly given in \cite{Bufetov_inf-deter}.

However, even for simple $g$, the explicit formula for the kernel of $\Pi^g$ turns out to be non-trivial. Our aim in this paper is to give explicit formulae for the kernel of the operator $\Pi^g$ for suitable $g$.  The kernels are obtained as the scaling limits of the Christoffel-Darboux kernels associated to Christoffel-Uvarov deformation of Jacobi orthogonal polynomial ensembles.

\subsubsection{Formulation of the main result}

Let $f_1, \cdots, f_n$ be complex-valued functions on an interval admitting $n-1$ derivatives.  We write $W(f_1, \dots, f_n)$ for the Wronskian of $f_1, \dots, f_n$, which, we recall, is defined by the formula $$W(f_1, \cdots, f_n) (t) = \left| \begin{array}{cccc} f_1(t) & f_1'(t) & \cdots & f_1^{(n-1)}(t)  \\  f_2(t) & f_2'(t) & \cdots & f_2^{(n-1)}(t)  \\ \vdots & \vdots & \ddots & \vdots \\  f_n(t) & f_n'(t) & \cdots & f_n^{(n-1)}(t) \end{array} \right|. $$
For $s^{\prime}>-1$, we write  $$J_{s^{\prime}, y} (t) \stackrel{\mathrm{def}}{=} J_{s^{\prime}} (t \sqrt{y}), \quad K_{s^{\prime}, v_j} (t) \stackrel{\mathrm{def}}{=} K_{s^{\prime}} (t \sqrt{v_j}),$$ where
 $ J_{s^{\prime}}$ stands for the Bessel function, $K_{s^{\prime}}$ for the modified Bessel function of the second kind.
 The main  result of this paper is given by the following
\begin{thm}\label{main-result}
Let $s \le -1$ and let $m$ be any natural number  such that $ s + m > -1$. Assume that $v_1, \dots, v_m$ are distinct positive real numbers. Then for the function \begin{align}\label{main-g} g (x) = \prod_{j = 1}^m \frac{4/x}{4/x + v_j} = \prod_{j= 1}^m \frac{4}{4 + v_j x}, \end{align} the kernel $\Pi^g$  is given by the formula \begin{align*} \Pi^g(x, x') =  \frac{1}{2} \cdot \frac{\left| \begin{array}{cc} A^{(s + m, v)} (1, 4/x) & B^{(s + m, v)} (1, 4/ x)   \vspace{3mm}\\ A^{(s + m, v)} (1, 4/x') & B^{(s + m, v)} (1, 4/ x') \end{array} \right|}{  \prod_{j = 1}^m \sqrt{(v_j + 4/x) (v_j + 4/x')} \cdot [C^{(s+m, v )}(1)]^2 \cdot (x' - x)}, \end{align*} where \begin{align*} A^{(s+m , v)} (t, y) = W ( K_{s+m, v_1}, \dots, K_{s+m, v_m}, J_{s+m, y}) (t) , \end{align*} \begin{align*} B^{(s+m, v)} (t, y) =  \frac{\partial  A^{(s+m, v)}}{\partial t} (t, y) , \end{align*} \begin{align*} C^{(s+m, v)} (t) = W ( K_{s+m, v_1}, \dots, K_{s+m, v_m}) (t).\end{align*}

\end{thm}

\begin{rem}
When $s > -1$,  the above theorem still holds for any $m \ge 1$. In this case, for the same $g$ as given in \eqref{main-g},  by results of \cite{Bufetov_inf-deter}, the kernel $\Pi^g$  obtained above is  the kernel for the operator of othogonal projection from $L_2(\R_{+}, \text{Leb})$ onto the subspace $ \sqrt{ g}  \mathrm{Ran}  J^{(s)}$ (here, with a slight abuse of notation, we let $J^{(s)}$ be the operator of orthogonal projection with kernel given in \eqref{bessel-kernel-mod}). Even in this case, however, the only way we can derive the explicit formula, given above,  for the kernel $\Pi^g$  is by using the method of scaling limits.
\end{rem}

\subsection{Organization of the paper}
The remainder of the paper is organized as follows.
Section 2 is devoted to some preliminary Mehler-Heine type asymptotics for Jacobi polynomials, 
these asymptotics will be used in the explicit calculations of the scaling limits in section 4.

In Section 3, we show that, for  three kinds of auxiliary functions $g$,  the scaling limits of the Christoffel-Darboux kernels for the Christoffel-Uvarov deformations of Jacobi orthogonal polynomial ensembles coincide with the kernels $\Pi^g$ which generate the determinantal probability given in \eqref{deter-proba}.

In section 4, we continue the study of the three kinds of auxiliary functions $g$. In case I, we illustrate how we calculate the scaling limits, the obtained scaling limits are the kernels for the determinantal process which are deformations of the Bessel point process of Tracy and Widom. The main formulae in Theorem \ref{main-result} will follow from the formulae obtained in case II, given in Theorem \ref{thm-case2-1} after change of variables $ z \rightarrow 4/x$.

{\bf {Acknowledgements.}}
Grigori Olshanski posed the problem to us, and we are greatly indebted to him.
We are deeply grateful to Alexei M. Borodin, who suggested to use the Christoffel-Uvarov deformations
of Jacobi orthogonal polynomial ensembles. We are deeply grateful to Alexei Klimenko for useful discussions.

The authors are supported by A*MIDEX project (No. ANR-11-IDEX-0001-02), financed by Programme ``Investissements d'Avenir'' of the Government of
the French Republic managed by the French National Research Agency (ANR).
A. B. is also supported in part
by the Grant MD-2859.2014.1 of the President of the Russian Federation,
by the Programme ``Dynamical systems and mathematical control theory''
of the Presidium of the Russian Academy of Sciences, by the ANR
 under the project "VALET'' of the Programme JCJC SIMI 1,
and by the
RFBR grants 11-01-00654, 12-01-31284, 12-01-33020, 13-01-12449 .
Y. Q. is supported in part by the ANR grant 2011-BS01-00801.

\section{Preliminary asymptotic formulae.}
\subsection{Notation} If $A, B$ are two quantities depending on the same variables, we write $A\asymp B$ if there exist two absolute values $c_1, c_2 > 0$ such that $c_1 \le \left| \frac{A}{B}\right| \le c_2$.  When $A$ and $B$ positive quantities, we write $A\lesssim  B$, if there exists an absolute value $c> 0$ such that $ A  \le c  B$.

For $\alpha, \beta> -1$, we denote the Jacobi weight on $(-1, 1)$ by  $$w_{\alpha, \beta}(t) = (1- t)^\alpha ( 1 + t)^\beta.$$  The associated Jacobi polynomials are denoted by $P_n^{(\alpha, \beta)}$. The leading coefficient of $P_n^{(\alpha, \beta)}$ is denoted by $k_n^{(\alpha, \beta)}$ and $h_n^{(\alpha, \beta)} : =\int [P_n^{(\alpha, \beta)} (t) ]^2 w_{\alpha, \beta}(t) dt $. When $\alpha = s, \beta = 0$, we will always omit $\beta$ in the notation: so $w_{s, 0}$ will be denoted by $w_s$,  $P_n^{(s, 0)}$ will be denoted by $P_n^{(s)}$ and the quantity $\Delta_{Q, n}^{(s, 0; \ell)}$ defined in the sequel will be denoted by  $\Delta_{Q, n}^{(s ; \ell)}$, etc.

  Given a sequence  $(f^{(\alpha, \beta)}_n)_{n = 0}^\infty$  of functions depending on $\alpha, \beta$, we define the differences of the sequence by $$ \Delta_{f, n}^{(\alpha, \beta; \,0)} : = f_n^{(\alpha, \beta)}, \quad \text{ and  for } \ell \ge 0, \Delta_{f, n}^{(\alpha, \beta; \,\ell + 1) } : = \Delta_{f, n+1}^{(\alpha, \beta; \,\ell) } - \Delta_{f, n}^{(\alpha, \beta; \,\ell)} . $$ By convention, we set $\Delta_{f, n}^{(\alpha, \beta; \, -1)} \equiv 0.$

In what follows, $\kappa_n$  always stands for a sequence of natural numbers such that $$\lim_{n \to \infty} \frac{\kappa_n}{n} = \kappa > 0. $$ Typical such sequences are given by $\kappa_n = \lfloor \kappa n \rfloor. $

\subsection{Asymptotics for Higher Differences of Jacobi Polynomials.}
In this section, we establish some asymptotic formulae for higher differences of Jacobi polynomials $\Delta_{P, n}^{(\alpha, \beta; \,\ell)} .$

\begin{lem}
For $\ell \ge 0$ and $n \ge 1$, we have
\begin{align}\label{recursion1} \begin{split} (n+1) \Delta_{P, n}^{( \alpha, \beta; \,\ell+1)} (x)  + \ell \Delta_{P, n+1}^{(\alpha, \beta; \,\ell)} (x)  + \ell (1-x) \Delta_{P, n+1}^{( \alpha+1, \beta; \ell-1 )} (x)  \\ + (n + \frac{\alpha + \beta}{2} + 1) (1-x) \Delta_{P, n}^{( \alpha + 1, \beta; \ell)} (x)= \alpha \Delta_{P, n}^{ (\alpha, \beta; \,\ell)} (x).
\end{split}
\end{align}
\end{lem}

\begin{proof}
When $\ell = 0$, identity \eqref{recursion1} is reduced to known formula (cf. \cite[4.5.4]{Szego-OP}): \begin{align}\label{recursion-difference} \begin{split} & (n + \frac{\alpha + \beta}{2} + 1)(1-x) P_n^{(\alpha + 1, \beta)} (x) \\ = & (n + 1) (P_n^{(\alpha, \beta)}(x) - P_{n+1}^{(\alpha, \beta)}(x) )  + \alpha P_n^{(\alpha, \beta)} (x).
\end{split} \end{align} Now assume that identity \eqref{recursion1} holds for an integer $\ell$ and for all $n\ge1$.  In particular, substituting $n+1$ for $n$,  we have \begin{align}\label{n+1}  \begin{split} & (n+2) \Delta_{P, n+1}^{( \alpha, \beta; \,\ell+1)} (x)  + \ell \Delta_{P, n+2}^{( \alpha, \beta; \,\ell )}   (x)+ \ell (1-x) \Delta_{P, n+2}^{( \alpha+1, \beta; \ell-1)} (x)   \\ & +  (n + \frac{\alpha + \beta}{2} + 2) (1-x) \Delta_{P, n+1}^{( \alpha + 1, \beta; \ell) } (x)= \alpha \Delta_{P, n+1}^{( \alpha, \beta; \,\ell )}(x). \end{split} \end{align} Then (\ref{n+1}) $-$ (\ref{recursion1}) yields that \begin{align*} & (n+1) \Delta_{P, n}^{ (\alpha, \beta; \,\ell+2)} (x)  + (\ell+1) \Delta_{P, n+1}^{( \alpha, \beta; \,\ell+1)} (x)  + (\ell +1)(1-x) \Delta_{P, n+1}^{( \alpha+1, \beta; \ell)} (x)    \\ & + (n + \frac{\alpha + \beta}{2} + 1) (1-x) \Delta_{P, n}^{( \alpha + 1, \beta; \ell + 1)} (x)= \alpha \Delta_{P, n}^{ (\alpha, \beta; \,\ell+1)} (x). \end{align*}
 Thus \eqref{recursion1} holds for $\ell + 1$ and all $n \ge 1$. By induction, identity \eqref{recursion1} holds for all $\ell \ge 0$ and all $n \ge 1$.
 \end{proof}

The classical Mehler-Heine theorem (\cite[p.192]{Szego-OP}) says that for $z \in \C\setminus \{ 0\}$,  \begin{eqnarray} \label{MH}\lim_{n \to \infty} n^{-\alpha} P_n^{(\alpha, \beta)} \Big(1- \frac{z}{2 n^2} \Big) = 2^{\alpha} z^{-\frac{\alpha}{2}} J_\alpha (\sqrt{z}).\end{eqnarray} This formula holds uniformly for $z$ in a simply connected compact subset of $ \C\setminus \{0\}$.

Applying the above asymptotics, we have
\begin{prop}\label{jacobi-asymp}
In the regime $ x^{(n)} = 1 - \frac{z}{2 n^2},$ for $\ell \ge 0$, we have
\begin{eqnarray}\label{asymp-gen}\lim_{n \to \infty} n^{\ell-\alpha} \Delta_{P, \kappa_n}^{( \alpha, \beta; \,\ell)} (x^{(n)}) = 2^\alpha z^{\frac{\ell-\alpha}{2}} J_\alpha^{(\ell)} (\kappa \sqrt{z}).\end{eqnarray} The formula holds uniformly in $\kappa$ and $z$ as long as $\kappa$ ranges in a compact subset of $(0, \infty)$ and $z$ ranges in a compact simply connected subset of $\C\setminus\{0\}$.
\end{prop}

\begin{proof}
When $\ell = 0$, identity (\ref{asymp-gen}) is readily reduced to the Mehler-Heine asymptotic formula \eqref{MH} and the uniform convergence. Now assume identity \eqref{asymp-gen} holds for $0, 1, \cdots, \ell$, then by \eqref{recursion1}, we have \begin{flalign} \label{induction} & \lim_{n \to \infty} n^{\ell + 1 - \alpha} \Delta_{P, k_n}^{ (\alpha, \beta; \,\ell+1)} (x^{(n)}) \\  = &  - \frac{\ell}{\kappa} \cdot 2^{\alpha} z^{\frac{\ell - \alpha}{2}}J_\alpha^{(\ell)} (\kappa \sqrt{z}) - \frac{\ell}{\kappa} \cdot \frac{z}{2} 2^{\alpha + 1} z^{\frac{\ell-1 - (\alpha+1)}{2}} J_{\alpha + 1}^{(\ell-1)} (\kappa \sqrt{z}) \nonumber  \\ & - \frac{z}{2} 2^{\alpha + 1} z^{\frac{\ell- (\alpha +1)}{2}} J_{\alpha + 1} ^{(\ell)} ( \kappa \sqrt{z}) + \frac{\alpha}{\kappa} 2^\alpha z^{\frac{\ell- \alpha}{2}} J_\alpha^{(\ell)}(\kappa \sqrt{z}) \nonumber \\ = & 2^\alpha z^{\frac{\ell + 1 - \alpha}{2}} \Big[  - \ell \cdot \frac{J_\alpha^{(\ell)}(\kappa \sqrt{z})}{ \kappa \sqrt{z}} - \ell \cdot \frac{J_{\alpha+1}^{(\ell - 1)}(\kappa \sqrt{z})}{ \kappa \sqrt{z}} - J_{\alpha + 1}^{(\ell)} (\kappa \sqrt{z}) + \alpha \frac{J_\alpha^{(\ell)}(\kappa\sqrt{z})}{\kappa \sqrt{z}}\Big]. \nonumber\end{flalign}

From the known recurrence relation (cf. \cite[9.1.27]{Ab})\begin{align} \label{differential-relation-J} J_\alpha'(z) = - J_{\alpha  + 1} (z)  + \frac{\alpha}{z} J_\alpha(z),\end{align} by induction on $\ell$, one readily sees that,  for all $\ell \ge 1$, \begin{eqnarray} \label{bessel-der} z \Big[  J_\alpha^{(\ell + 1)} (z) + J_{\alpha  + 1}^{(\ell )}(z)\Big] = (\alpha - \ell ) J_\alpha^{(\ell)} (z) - \ell J_{\alpha+1}^{(\ell-1)}(z). \end{eqnarray} Identity \eqref{asymp-gen} for $\ell + 1$ follows from \eqref{induction} and \eqref{bessel-der}, thus the proposition is proved by induction on $\ell$.
\end{proof}

We will also need the asymptotics for the derivative of the differences of Jacobi polynomials. The derivative of the Jacobi polynomials can be expressed in Jacobi polynomials with different parameters, more precisely, we have \begin{align}\label{jacobi-der} \dot{P}_n^{(\alpha, \beta)} (t)  = \frac{d}{dt}\Big\{P_n^{(\alpha, \beta)} \Big\} (t) = \frac{1}{2} ( n + \alpha + \beta +1) P_{n - 1}^{(\alpha + 1, \beta +1)} (t).\end{align}Using this relation, we have

\begin{prop}\label{der-asymp}
In the regime $ x^{(n)} = 1 - \frac{z}{2 n^2},$ for $\ell \ge 0$, we have \begin{align*} \lim_{n \to \infty} n^{-2 + \ell - \alpha} \dot{\Delta}_{P,\kappa_n}^{ (\alpha, \beta; \,\ell)} (x^{(n)}) =  2^\alpha z^{\frac{-2 + \ell -\alpha}{2}} \widetilde{J}_{\alpha+1}^{(\ell)}( \kappa \sqrt{z}),\end{align*} where $\widetilde{J}_{\alpha + 1}(t) : =  t J_{\alpha +1}(t)$. The formula holds uniformly in $\kappa$ and $z$ as long as $\kappa$ ranges in a compact subset of $(0, \infty)$ and $z$ ranges in a compact simply connected subset of $\C\setminus\{0\}$.
\end{prop}

\begin{proof}
The relation \eqref{jacobi-der} can be written as $$ 2 \dot{\Delta}_{P, n}^{ (\alpha, \beta; \,0)} = (n + \alpha   + \beta + 1) \Delta_{P, n - 1}^{ (\alpha + 1, \beta + 1; 0)}.$$ From this formula, it is readily to deduce by induction that for all $\ell \ge 0$, \begin{align}\label{rel-der} 2 \dot{\Delta}_{P, n}^{( \alpha, \beta; \,\ell)}  = (n + \alpha  + \beta + 1) \Delta_{P, n- 1}^{( \alpha+ 1, \beta + 1; \ell)} + \ell \cdot \Delta_{P, n}^{( \alpha + 1, \beta + 1; \ell-1)}. \end{align} In view of Proposition \ref{jacobi-asymp} and  identity (\ref{rel-der}), we have  \begin{align*} & \lim_{n \to \infty} n^{-2 + \ell - \alpha} \dot{\Delta}_{P,n}^{ (\alpha, \beta; \,\ell)} (x^{(n)}) \\  = &     2^\alpha  z^{\frac{- 2 + \ell - \alpha}{2}} \Big[ \kappa  \sqrt{z} J_{\alpha + 1}^{(\ell)} ( \kappa \sqrt{z}) + \ell J_{\alpha + 1}^{(\ell-1)}( \kappa \sqrt{z})\Big]  \\ = &  2^\alpha z^{\frac{-2 + \ell -\alpha}{2}} \widetilde{J}_{\alpha+1}^{(\ell)}( \kappa \sqrt{z}). \end{align*} The last equality follows from Leibniz formula $$\Big(t J_{\alpha+1}(t)\Big)^{(\ell)} = t J_{\alpha + 1}^{(\ell)} (t) + \ell J_{\alpha + 1}^{(\ell - 1)} (t) .$$
\end{proof}

\subsection{Asymptotics for Higher Differences of Jacobi's Functions of the Second Kind.}

Let $Q_n^{(\alpha, \beta)} $ be the Jacobi's functions of second kind defined as follows. For  $ x \in \C \setminus [-1, 1]$,  $$ Q_n^{(\alpha, \beta)} (x) : = \frac{1}{2} (x - 1)^{- \alpha} (x + 1)^{-\beta} \int_{-1}^1 (1-t)^{\alpha} (1 + t )^{\beta} \frac{P_n^{(\alpha, \beta)} (t)}{x - t} dt. $$

\begin{prop}\label{asymp-Q}
Let $s > -1$ and  $r_n = \frac{w}{2n^2}$ with $ w > 0$. Then
$$\lim_{n \to \infty}  n^{-s}  Q_{\kappa_n}^{(s)} ( 1 + r_n ) =  2^s w^{- \frac{s}{2}} K_s( \kappa \sqrt{w}),$$ where $K_s$ is the modified Bessel function of second kind with order $s$. For any $\varepsilon > 0$, the convergence is uniform as long as $\kappa \in [\varepsilon, 1]$ and $w$ ranges in a bounded simply connected subset of $\C\setminus \{0\}$.
\end{prop}

\begin{proof}
We show the proposition when $\kappa_n = n$, the general case is similar.
Define $t_n$ by the formula $$ 1 + r_n = \frac{1}{2} \Big( t_n + \frac{1}{t_n} \Big), \quad | t_n | < 1.$$ By definition, we have $$\lim_{n \to \infty} n ( 1 - t_n) = \sqrt{w}. $$ We now use the integral representation for the Jacobi function of the second kind (cf.  \cite[4.82.4]{Szego-OP}). Write \begin{align*}  Q_n^{(s)} ( 1 + r_n) =   \frac{1}{2} \Big( \frac{4t_n}{ 1 - t_n}\Big)^s  & \int_{-\infty}^{\infty} \Big( ( 1+ t_n ) e^{\tau} + 1 - t_n \Big)^{-s} \times \\ & \times \Big(  1 + r_n + (2r_n + r_n^2)^{\frac{1}{2}} \cosh \tau \Big)^{- n - 1} d\tau .\end{align*} Taking $ n \to \infty$ and using the integral representation for the modified Bessel function(cf. \cite[9.6.24]{Ab}), we see that  \begin{align*} &    \lim_{n \to \infty} n^{-s} Q_n^{(s)} ( 1 + r_n)  =   2^{s - 1} w^{-\frac{s}{2}} \int_{-\infty}^{\infty} e^{-s \tau - \sqrt{w} \cosh \tau} d \tau  \\ & = 2^{s - 1} w^{-\frac{s}{2}} \int_{-\infty}^{\infty} e^{- \sqrt{w} \cosh \tau}  \cosh (s\tau) d \tau \\ & =  2^{s } w^{-\frac{s}{2}} \int_{0}^{\infty} e^{- \sqrt{w} \cosh \tau}  \cosh (s\tau) d \tau = 2^s w^{- \frac{s}{2}} K_s(  \sqrt{w}). \end{align*}
\end{proof}

\begin{prop}\label{asymp-diff-Q}
In the same condition as in Proposition \ref{asymp-Q}, we have for all $ \ell \ge 0$,  \begin{eqnarray}\label{asymp-Q-diff}  \lim_{n \to \infty} n^{\ell - s} \Delta_{Q, \kappa_n}^{(s; \, \ell)} ( 1 + r_n ) = 2^s w^{\frac{\ell-s}{2}}K_s^{(\ell)} (\kappa \sqrt{w}), \end{eqnarray} where $K_s^{(\ell)}$ is the $\ell$-th derivative of the modified Bessel function of second kind $K_s$. Moreover,  For any $\varepsilon > 0$, the convergence is uniform as long as $\kappa \in [\varepsilon, 1]$ and $w$ ranges in a bounded simply connected subset of $\C\setminus \{0\}$.
\end{prop}

\begin{proof}

It suffices to prove the proposition in the case $ \kappa_n = n$. The general case can easily be deduced from this special case by using the uniform convergence.

From the identity \eqref{recursion-difference} we obtain $$ (n+1) \Delta_{Q, n}^{(s; \, 1)}  (x) + (n + \frac{s}{2} + 1) (x-1)\Delta_{Q, n}^{(s + 1; \,  0)} (x) = s \Delta_{Q,n}^{(s;\, 0)} (x). $$ By induction, it is readily to write \begin{eqnarray} \label{rec-sec-diff} & & (n+1) \Delta_{Q, n}^{(s; \,  \ell+1)} (x)  + \ell \Delta_{Q, n+1}^{(s; \, \ell)} (x)  +   \ell  (x-1) \Delta_{Q, n + 1}^{(s+1; \, \ell-1)} (x) \\ & &  + ( n + \frac{s}{2} + 1) (x-1)\Delta_{Q, n}^{(s +1; \,  \ell)} (x)   =  s \Delta_{Q, n}^{(s; \, \ell)} (x) ,\nonumber\end{eqnarray} for all $\ell \ge 0$ where by convention, we set $\Delta_{Q, n}^{(s; \, -1)}: = 0$  .

Using the formula (\cite[9.6.26]{Ab}) \begin{align}\label{differential-relation-K} K_s'(t) = - K_{s+1}(t)  + \frac{s}{t} K_s(t), \end{align} we can show that  for $\ell \ge 1$, \begin{eqnarray}\label{derivative-sec-bessel} t \Big[  K_s^{(\ell+1)} (t) + K_{s+1}^{(\ell)} (t)  \Big] = (s-\ell) K_s^{(\ell)}(t)  - \ell K_{s+1}^{(\ell-1)} (t). \end{eqnarray}

Proposition \ref{asymp-Q} says that the equation \eqref{asymp-Q-diff} holds for $\ell  = 0$. Now assume \eqref{asymp-Q-diff} holds for $0, 1, \cdots, \ell$. By (\ref{rec-sec-diff}), we have \begin{eqnarray*} & & \lim_{n \to \infty} n^{\ell+ 1 - s} \Delta_{Q, n}^{(s; \, \ell+1)} ( 1 + r_j^{(n)} )  \\ &  =&   - \ell  \cdot 2^s w_j^{\frac{\ell - s}{2}} K_s^{(\ell)}( \sqrt{w_j})  -  \ell \cdot \frac{w_j}{2} 2^{s+1} w_j^{\frac{\ell - s -2}{2}} K_{s+1}^{(\ell - 1)} ( \sqrt{w_j}) \\ & & -  \frac{w_j}{2}  \cdot 2^{s+1} w_j^{\frac{\ell - s - 1}{2}} K_{s+1}^{(\ell)}( \sqrt{w_j}) + s \cdot  2^s w_j^{\frac{\ell-s }{2}} K_s^{(\ell)}(\sqrt{w_j}) \\  & = & 2^s w_j^{\frac{\ell-s}{2} } \Big[ (s - \ell ) K_s^{(\ell)}(\sqrt{w_j}) - \ell K_{s + 1}^{(\ell  - 1)} (\sqrt{w_j})  - \sqrt{w_j} K_{s + 1}^{(\ell)} (\sqrt{w_j}) \Big]  \\ & = & 2^s w_j^{\frac{\ell +1-s}{2}} K_s^{(\ell+1)}(\sqrt{w_j}).   \end{eqnarray*} This completes the proof.
\end{proof}

\subsection{Asymptotics for Higher Differences of $R_n^{(\alpha, \beta)}$.}

\begin{defn}
Define for $ x \in \C \setminus [-1, 1]$,  $$R_n^{(\alpha, \beta)}(x) : = (x - 1)^{-\alpha} (x + 1)^{- \beta } \int_{-1}^1 \frac{P_n^{(\alpha, \beta)} (t) }{(x - t)^2 } ( 1 - t)^{\alpha} ( 1 + t)^{\beta} d t. $$
\end{defn}

\begin{defn}
For any $s \in \R$, define $$L_s(x) : = s K_{s } (x) -  \frac{x K_{s-1}(x) +  x K_{s + 1} (x) }{2}. $$
\end{defn}

\begin{prop}
Let $ s> -1$, and $\gamma^{(n)} = \frac{u}{2n^2}$ with $ u > 0$. Then we have $$\lim_{n \to \infty}  n^{-2- s}R_{\kappa_n}^{(s)} (  1 + \gamma^{(n)} )   =  2^{\frac{2s + 3}{2}} \cdot u^{-\frac{s + 2}{2}} L_{s}(\kappa \sqrt{u}).$$ Moreover, for any $\varepsilon > 0$,  the convergence is uniform as long as $ \kappa \in [\varepsilon, 1 ]$ and $ u  $ ranges in a compact subset of $(0, \infty)$.
\end{prop}

\begin{proof}
The uniform convergence can be derived by a careful look at  the following proof. By this uniform convergence, it suffices to show the proposition for $\kappa_n = n $.

Define $z$ by the formula $$ x = \frac{1}{2} \Big( z + \frac{1}{z} \Big), \quad | z | < 1.$$ By the integral representation for the Jacobi function of the second kind (\cite[4.82.4]{Szego-OP}), we have \begin{align*}  Q_n^{(s)} ( x ) =   \frac{1}{2} \Big( \frac{4z}{ 1 - z}\Big)^s  & \int_{-\infty}^{\infty} \Big( ( 1+ z ) e^{\tau} + 1 -z \Big)^{-s} \times \\ & \times \Big(  x + (x^2 - 1)^{\frac{1}{2}} \cosh \tau \Big)^{- n - 1} d\tau .\end{align*} Denote $$ \widehat{Q}^{(s)}_n(x) : = 2 ( x-1)^s Q_n^{(s)} (x) = \int_{-1}^1 \frac{P_n^{(s)} (t) }{ x - t } ( 1 - t)^s dt .$$ Then \begin{align}\label{QR}\left[\frac{d}{dx} \widehat{Q}_n^{(s)} \right] (x) =  - \int_{-1}^1 \frac{P_n^{(s)}(t) }{(x - t)^2} ( 1- t)^s d t  = -  ( x - 1)^s R_n^{(s)} ( x) .\end{align} We have \begin{align*} \widehat{Q}^{(s)}_n(x)  & =   2^s \int_{-\infty}^{\infty} \left( 1 + \sqrt{\frac{x + 1}{x  - 1}} e^{\tau} \right)^{-s}  \Big(  x + (x^2 - 1)^{\frac{1}{2}} \cosh \tau \Big)^{- n - 1} d\tau.\end{align*}
Hence $$ \left[\frac{d}{dx} \widehat{Q}^{(s)}_n\right] (x) = T_1^{(n)}(x)  - T_{2}^{(n)} (x) , $$ where \begin{align*}  T_1^{(n)} (x) =   \frac{s \cdot 2^s }{(x-1)^2} \sqrt{\frac{x - 1}{ x + 1}}  \int_{-\infty}^{\infty} e^{\tau}  & \left( 1 + \sqrt{\frac{x + 1}{x  - 1}} e^{\tau} \right)^{-s-1} \times \\ & \times  \Big(  x + (x^2 - 1)^{\frac{1}{2}} \cosh \tau \Big)^{- n - 1} d\tau \end{align*} and \begin{align*} T_2^{(n)} (x) = (n+1) 2^s \int_{-\infty}^{\infty}  & \left( 1 + \sqrt{\frac{x + 1}{x  - 1}} e^{\tau} \right)^{-s}  \Big(  x + (x^2 - 1)^{\frac{1}{2}} \cosh \tau \Big)^{- n - 2} \times \\ & \times  ( 1 + \frac{x}{\sqrt{ x^2 - 1}} \cosh \tau ) d\tau .\end{align*} We have \begin{align*}  \lim_{n \to \infty} n^{s -2} T_1^{(n)}( 1 + \gamma^{(n)})&  =\sqrt{2} s \cdot u^{\frac{s-2}{2}}\int_{-\infty}^\infty e^{-s \tau - \sqrt{u}\cosh \tau} d \tau \\ & =  2 \sqrt{2} s \cdot u^{\frac{s-2}{2} } K_s (\sqrt{u}).  \end{align*} \begin{align*} \lim_{n \to \infty} n^{s -2} T_2^{(n)} (1 + \gamma^{(n)} )   & =   \sqrt{2} u^{  \frac{s - 1}{2}}  \int_{-\infty}^\infty    e^{-s \tau} e^{- \sqrt{u} \cosh \tau} \cosh \tau  d \tau \\ & =   \sqrt{2 } u^{  \frac{s - 1}{2}}  \left(K_{s + 1}(\sqrt{u}) + K_{s-1}(\sqrt{u})\right).  \end{align*} Hence \begin{align*} & \lim_{n \to \infty} n^{s -2} \left[ \frac{d}{d x} \widehat{Q}_n^{(s)} \right] ( 1 + \gamma^{(n)})  \\ = &  2 \sqrt{2} u^{\frac{s - 2}{2}} \left( s K_s(\sqrt{u}) - \frac{ \sqrt{u}K_{s + 1} (\sqrt{u}) +  \sqrt{u}K_{s-1}(\sqrt{u}) }{2}\right) \\ = & 2 \sqrt{2} u^{\frac{s-2}{2}} L_{s} (\sqrt{u}).   \end{align*} In view of \eqref{QR}, we prove the desired result.
\end{proof}

{\flushleft \bf Remark. } We have the following relations \begin{align} \label{L-relation}  L_s'(x) = - L_{s + 1}(x) + \frac{s }{x } L_s(x), \end{align}  \begin{align} \label{L-der} x \Big[  L_s^{(\ell + 1)} (x) + L_{s + 1}^{(\ell )}(x)\Big] = (s - \ell ) L_s^{(\ell)} (x) - \ell L_{s+1}^{(\ell-1)}(x). \end{align} Let us for example show \eqref{L-relation}. The validity of \eqref{L-der} can be verified easily by induction on $\ell$.  We have \begin{align*} L_s'(x)  = & s K_s'(x)  - \frac{x K_{s - 1}' (x) + K_{s-1}(x) +  x K_{s +1}'(x) + K_{s +1} (x)  }{2} \\ = &  -s K_{s + 1} (x) + \frac{s^2}{x} K_s(x)  - \frac{- x K_s(x) + (s - 1) K_{s - 1}(x)}{2} \\ & -\frac{  -x  K_{s +2} (x) + (s+ 1) K_{s + 1}(x) }{2} \\ = & - \left( (s + 1) K_{s + 1}(x) - \frac{x K_s(x) +  x K_{ s+2} (x) }{2}\right)   \\ &  +  \frac{s}{x}  \left( s K_{s } (x) -  \frac{x K_{s-1}(x) +  x K_{s + 1} (x) }{2} \right) \\ = & - L_{s + 1}(x) + \frac{s }{x} L_s(x) . \end{align*}

\begin{prop}
Let $ s> -1$, and $\gamma^{(n)} = \frac{u}{2n^2}$ with $ u > 0$. Then for $\ell \ge 0$, we have $$\lim_{ n \to \infty}  n^{\ell - s-2}  \Delta_{R, \kappa_n}^{(s; \, \ell)} ( 1 + \gamma^{(n)}) =  2^{\frac{2s + 3}{2}}  \cdot u^{\frac{ \ell -  s-2  }{2}} L_{s}^{(\ell)}(\kappa \sqrt{u}) . $$  Moreover, for any $\varepsilon > 0$,  the convergence is uniform as long as $ \kappa \in [\varepsilon, 1 ]$ and $ u  $ ranges in a compact subset of $(0, \infty)$.
\end{prop}

\begin{proof}
Again, we show the proposition only for $\kappa_n = n$. The formula holds for $\ell = 0$. Assume that the formula holds for all $0 , 1, \cdots, \ell$, we shall show that it holds for $\ell + 1$. By similar arguments as that for $ \Delta_{Q, n}^{(s; \, \ell)}$, we can easily obtain that  \begin{align*}  & (n+1) \Delta_{R,  n}^{(s; \,  \ell+1)} (x)  + \ell \Delta_{R,  n+1}^{(s; \, \ell)} (x)  +   \ell  (x-1) \Delta_{R, n + 1}^{(s+1; \, \ell-1)} (x) \\ &  + ( n + \frac{s}{2} + 1) (x-1)\Delta_{R, n}^{(s +1; \,  \ell)} (x)   =  s \Delta_{R, n}^{(s; \, \ell)} (x) . \end{align*} Hence \begin{align*} & \lim_{n \to \infty} n^{(\ell +1) - s-2} \Delta_{R, n}^{(s; \, \ell + 1)} ( 1 +  \gamma^{(n)}) \\   = &  (s- \ell)     2^{\frac{2s + 3}{2}} \cdot u^{\frac{ \ell - s - 2}{2}} L_{s}^{(\ell)}( \sqrt{u})  -   \ell   \frac{u }{2} \cdot 2^{\frac{2s + 5}{2}}  \cdot u^{\frac{ \ell - s - 4}{2}} L_{s + 1}^{(\ell-1)}( \sqrt{u})  \\ & - \frac{u }{2} \cdot 2^{\frac{2s + 5}{2}}  \cdot u^{\frac{ \ell - s - 3}{2}} L_{s+1}^{(\ell)}( \sqrt{u})  \\ = & 2^{\frac{2s + 3}{2}} \cdot u^{\frac{ \ell - s -1 }{2}} \left(  \frac{s - \ell}{\sqrt{u}}  L_s^{(\ell)} (\sqrt{u})   - \frac{\ell}{\sqrt{u}} L_{ s + 1}^{(\ell-1)} (\sqrt{u}) - L_{s+1}^{(\ell)} (\sqrt{u})  \right) \\  = &  2^{\frac{2s + 3}{2}}  \cdot u^{\frac{(\ell +1) - s -2}{2}}  L_{s + 1}^{(\ell + 1)} (\sqrt{u}).
\end{align*}

\end{proof}


\section{Bessel Point Processes as Radial Parts of Pickrell Measures on Infinite Matrices}

\subsubsection{Radial parts of Pickrell measures and the infinite Bessel point processes}

Following Pickrell, we introduce a map $$\mathfrak{rad}_n: \text{Mat}(n, \C) \rightarrow \R_{+}^n $$ by the formula $$ \mathfrak{rad}_n(z) = (\lambda_1 (z^*z), \dots, \lambda_n(z^*z)) .$$ Here $ (\lambda_1 (z^*z), \dots, \lambda_n(z^*z))$ is the collection of the eigenvalues of the matrix $z^*z$ arranged in non-decreasing order.

The radial part of the Pickrell measure $\mu_n^{(s)}$ is defined as $(\mathfrak{rad}_n )_{* } \mu_n^{(s)}$. Note that, since finite-dimensional unitary groups are compact, even for  $s \le -1$, if $n + s > 0$, then the radial part of $\mu_n^{(s)}$ is well-defined.

Denote $d\lambda$ the Lebesgue measure on $\R_{+}^n$, then the radial part of the measure $\mu_n^{(s)}$ takes the form $$ \text{const}_{n, s} \cdot \prod_{i< j} (\lambda_i - \lambda_j)^2 \cdot \frac{1}{( 1+ \lambda_i)^{2n + s}} d\lambda.$$ After the change of variables $$u_i = \frac{\lambda_i -1}{ \lambda_i + 1}, $$ the radial part $(\mathfrak{rad}_n)_{*} \mu_n^{(s)} = (\mathfrak{rad}_n \circ \pi_n)_{*} \mu^{(s)}$ is a measure defined on $(-1, 1)^n$ by the formula \begin{align}   \label{jacobi-measure}  \text{const}_{n,s} \cdot \prod_{1\le i <   j \le n} (u_i - u_j)^2 \cdot \prod_{ i = 1}^n ( 1- u_i)^s du_i. \end{align}



For $ s > -1$,  the constant is chosen such that the measure \eqref{jacobi-measure} is a probability measure, it is the Jacobi orthogonal polynomial ensemble, a determinantal point process induced by the $n$-th Christoffel-Darboux projection operator for Jacobi polynomials. The classical Heine-Mehler asympotitics of Jacobi polynomials imply that these determinantal point processes, when rescaled with the scaling \begin{align}\label{rescal} u_i = 1 - \frac{y_i}{2n^2}, i = 1, \dots, n, \end{align} have as scaling limit the Bessel point process of Tracy and Widom \cite{Tracy-Widom94}, use the same notation as in \cite{Bufetov_inf-deter}, we denote this point process on $ (0, \infty)$ by $\widetilde{\mathbb{B}}^{(s)}$.

For $ s \le -1$, the scaling limit under the scaling regime \eqref{rescal} is an infinite determinantal measure $\widetilde{\mathbb{B}}^{(s)}$ on $\Conf((0, \infty))$.

In both cases, $\widetilde{\mathbb{B}}^{(s)}$ is closely related to the decomposition measure $\mathbb{B}^{(s)}$ for the Pickrell measure $\mu^{(s)}$:  the change of variable $y = 4/x$ reduces the decomposition measure $\mathbb{B}^{(s)}$ to  $\widetilde{\mathbb{B}}^{(s)}$.

\subsubsection{Christoffel-Uvarov deformations of Jacobi orthogonal polynomial ensembles and the scaling limits}

Now consider  a sequence of functions $ g^{(n)}:  (-1, 1) \rightarrow (0, 1]$ such that the measures $(1 - u)^s g^{(n)} (u) du $ on $(-1,1)$ have moments of all orders. On the cube $(-1, 1)^n$,  the probability measure $$ \text{const}_{n,s} \cdot \prod_{ 1 \le i < j \le n} (u_i - u_j)^2 \prod_{i = 1}^n ( 1-u_i)^s g^{(n)} (u_i) d u_i $$ gives a determinantal point process induced by the corresponding $n$-th Christoffel-Darboux projection. After change of variable \begin{align} \label{rescal2}  y = \frac{n^2 x - 1}{n^2 x + 1}, \end{align} this point process becomes \begin{align} \label{multi-change} \mathbb{P}_{\widetilde{g}^{(n)}}^{(s,n)} : = \frac{\Psi_{\widetilde{g}^{(n)}} \mathbb{B}^{(s,n )}}{ \int_{ \text{Conf} \big((0, + \infty)\big)} \Psi_{\widetilde{g}^{(n)}} d \mathbb{B}^{(s,n )} }, \end{align}
where  $ \mathbb{B}^{(s,n)}$ is the point process $(\mathfrak{rad}_n \circ \pi_n)_{*}\mu^{(s)}$ after the change of variable given in \eqref{rescal2} and $\widetilde{g}^{(n)}$ is the function on $(0, \infty)$ given by  \begin{align} \label{functions-g}\widetilde{g}^{(n)} ( x ) = g^{(n)} ( \frac{n^2 x-1}{n^2 x+ 1}).\end{align}

We shall need the following elementary lemma, whose routine proof is included for completeness.

\begin{lem}\label{DC-lem}
Let $(\Omega, m)$ be a measure space equipped with a $\sigma$-finite measure $m$. Given two sequence of positive integrable functions $(F_n)_{n=1}^\infty$ and $(f_n)_{n = 1}^{\infty}$ satisfying \begin{itemize} \item[(a)] for any $ n \in \N,$ $f_n \le F_n$.
\item[(b)] $\lim_{n \to \infty} f_n = f, a.e.   \text{ and  } \lim_{n \to \infty} F_n = F, a.e.. $
\item[(c)] $  \lim_{n \to \infty} \int F_n dm = \int F dm < \infty$ .
\end{itemize}
Then $$ \lim_{n \to \infty} \int f_n dm= \int f dm.$$
\end{lem}

\begin{proof}
By Fatou's lemma, we have \begin{align*}  \int f d m  \le \liminf_{n \to \infty} \int f_n dm.  \end{align*} Again by Fatou's lemma applied on the positive sequence $F_n - f_n$, we have \begin{align*}  \int ( F - f) dm \le \liminf_{n \to \infty} \int  (F_n - f_n) dm    = \int F d m - \limsup_{n \to \infty} \int f_n dm. \end{align*}  Hence \begin{align*} \limsup_{n \to \infty} \int f_n dm \le \int f d m .  \end{align*} Combining these inequalities, we get the desired result.
\end{proof}

The following three kinds of auxiliary functions are considered: 
\begin{align} \label{g_I}g_I^{(n)} (u) =  \prod_{i = 1}^m \frac{(1 - \frac{w_i}{2n^2} - u)^2}{ (1 - u)^2 } , \, w_i \ne w_j ; \end{align}
\begin{align} \label{g_II} g_{II}^{(n)} (u) = \prod_{i  = 1}^m \frac{1 - u }{ 1 + \frac{v_i}{2n^2}  - u }, \, v_i \ne v_j ; \end{align}
\begin{align} \label{g_III}g_{III}^{(n)} (u) = \frac{(1-u)^2}{ (1 + \frac{v}{2n^2}  -u )^2} .\end{align}

Let $\widetilde{g}_I^{(n)} ( x )$ denote the function given by  $\widetilde{g}_I^{(n)} ( x ) = g^{(n)}_I ( \frac{n^2 x-1}{n^2 x+ 1}).$ Similarly, let $ \widetilde{g}_{II}^{(n)} ( x ) = g^{(n)}_{II} ( \frac{n^2 x-1}{n^2 x+ 1})$ and $\widetilde{g}_{III}^{(n)} ( x ) = g^{(n)}_{III} ( \frac{n^2 x-1}{n^2 x+ 1})$.

If $ \widetilde{g}^{(n)}  $ is one of the functions $\widetilde{g}_I^{(n)}$, $\widetilde{g}_{II}^{(n)}$ $\widetilde{g}_{III}^{(n)}$ , then there exists a positive function $ g: (0, \infty) \rightarrow [0, 1] $ and a constant $M > 0$  satisfying \begin{itemize}
\item[(a)] $\lim_{n \to \infty}\widetilde{g}^{(n)} (x) = g(x).$
\item[(b)] for any (finite or infinite)  sequence of positive real numbers $(x_i)_{i = 1}^N$, we have $$\prod_{ i = 1}^N \widetilde{g}^{(n)} (x_i) \le M \cdot \prod_{i = 1}^N g(x_i). $$
\item[(c)] for any sequence $\{(x_i^{(n)})_{1 \le i \le n }\}_{n=1}^\infty$ satisfying $x_i^{(n)} \ge 0$,  $$\lim_{n \to \infty} x_i^{(n)} = x_i \, \text{ and } \, \lim_{n \to \infty} \sum_{i = 1}^n x_i^{(n)}  = \sum_{i = 1}^\infty x_i< \infty,$$ we have $$ \lim_{n \to \infty} \prod_{ i = 1}^n \widetilde{g}^{(n)}  (x_i^{(n)}  ) = \prod_{ i = 1}^{\infty} g(x_i). $$
\end{itemize} The limiting functions are \begin{align}\label{asym-g_I} g_I (x)   = \prod_{i = 1}^m (1 - \frac{w_i }{4} x ) ^2  & \sim  1 - \left(\sum_{i } \frac{w_i}{2}\right) x , \text{ as } x \to 0+ ; \\ \label{asym-g_II}g_{II}(x)  = \prod_{i = 1}^m \frac{4}{4 + v_i x } & \sim 1 - \left(\sum_i \frac{v_i}{4} \right)x , \text{ as } x \to 0+  ; \\ \label{asym-g_III} g_{III}(x)  = \left(\frac{4}{4 + v x }\right)^m &  \sim 1 - m \cdot  \frac{v}{4} x , \text{ as } x \to 0+ . \end{align}

\begin{prop}\label{abstract-result}
Assume we are in one of the following situations: \begin{itemize}  \item [I.] $g^{(n)} = g_I^{(n)}$ and $ g = g_I$ with $ s- 2m > -1$, \item[II.] $ g^{(n)} = g_{II}^{(n)}$ and $g = g_{II}$ with $m + s > -1$, \item[III.] $ g^{(n)} = g_{III}^{(n)}$ and $g= g_{III} $ with $ m + s > -1$.  \end{itemize} Then the determinantal probability measure in    \eqref{multi-change} converges weakly in $\mathfrak{M}_{\textnormal{fin}} (\textnormal{Conf} ((0, + \infty)))$ to  \begin{align}  \mathbb{P}_{g}^{(s)} : = \frac{\Psi_{g} \mathbb{B}^{(s )}}{ \int_{ \textnormal{Conf} \left((0, \infty)\right)} \Psi_{g} d \mathbb{B}^{(s )} }.\end{align}
\end{prop}

\begin{proof}
We will use the notation in \cite{Bufetov_inf-deter}, where,  following \cite{BO-infinite-matrices}, it is proved that the measure $\mu^{(s)}$ is  supported on the subset $\text{Mat}_{\text{reg}} (\N, \C)$ for any $s \in \R$. By the remarks preceding the proposition, for any $ z \in \text{Mat}_{\text{reg}} (\N, \C)$, we have  $$\lim_{n \to \infty} \Psi_{\widetilde{g}^{(n)}} ( \mathfrak{r}^{(n)} (z) )  = \Psi_g(\mathfrak{r}^{\infty} (z)), $$ and $$\Psi_{\widetilde{g}^{(n)}} (\mathfrak{r}^{(n)} (z) ) \le  M \cdot \Psi_{g} (\mathfrak{r}^{(n)} (z) ). $$   Now take any bounded and continuous  function  $f$  on $\text{Conf}((0,  \infty))$, we have \begin{align*} \int f(X) d \mathbb{P}_{\widetilde{g}^{(n)}}^{(s,n)} (X)   = \frac{ \int_{\text{Mat}_{\text{reg}}(\N, \C)} f(\mathfrak{r}^{(n)} (z))  \Psi_{\widetilde{g}^{(n)}} (\mathfrak{r}^{(n)} (z) )  d \mu^{(s)} (z)}{\int_{\text{Mat}_{\text{reg}}(\N, \C)}  \Psi_{\widetilde{g}^{(n)}} (\mathfrak{r}^{(n)} (z) )  d \mu^{(s)} (z) }. \end{align*} By Lemma \ref{DC-lem}, it suffices to show that \begin{align}\label{wanted} \lim_{n \to \infty}\int_{\text{Mat}(\N, \C)}  \Psi_{g} (\mathfrak{r}^{(n)} (z) )  d \mu^{(s)} (z)  = \int_{\text{Mat}(\N, \C)}  \Psi_{g} (\mathfrak{r}^{(\infty)} (z) )  d \mu^{(s)} (z) . \end{align}

If $s> -1$, the measure $\mu^{(s)}$ is a probability measure, by dominated convergence theorem, the equality \eqref{wanted} holds.

If $s \le -1$, the measure $\mu^{(s)}$ is infinite. The radial part of $\mu^{(s)}$ is an infinite determinantal process which corresponds to a finite-rank perturbation of determinantal probability measures as described in \S 5.2 in \cite{Bufetov_inf-deter}.  By using the asympotic formulae \eqref{asym-g_I}, \eqref{asym-g_II} and \eqref{asym-g_III} respectively in these three cases, we can check that the conditions of  Proposition 3.6 in \cite{Bufetov_inf-deter} are satisfied, for instance, let us check the following condition \begin{align}\label{UI-condition}  \lim_{n \to \infty} \tr \sqrt{1 - g} \Pi^{(s, n)} \sqrt{1-g} = \tr \sqrt{1- g} \Pi^{(s)} \sqrt{1-g},\end{align}  where $\Pi^{(s, n)} $ is the orthogonal projection onto the subspace $L^{(s + 2 n_s, n - n_s)}$ described in \S 5.2.1 in \cite{Bufetov_inf-deter}.  Combining the estimates given in Proposition 5.11 and Proposition 5.13 in \cite{Bufetov_inf-deter},  the integrands appeared in $$\tr \sqrt{1-g} \Pi^{(s,n)} \sqrt{1-g}$$ are uniformly integrable, hence by the Heine-Mehler classical asymptotics, the equality \eqref{UI-condition} indeed holds.   Now by Corollary 3.7 in \cite{Bufetov_inf-deter}, we have   \begin{align*} \frac{\Psi_g \mathbb{B}^{(s,n)}}{ \int_{\textnormal{Conf}\big( (0,  \infty) \big)} \Psi_g  d \mathbb{B}^{(s,n)}} \rightarrow  \frac{\Psi_g \mathbb{B}^{(s)}}{ \int_{\textnormal{Conf}\big( (0,  \infty) \big)} \Psi_g  d \mathbb{B}^{(s)}} . \end{align*} It follows that \begin{align}\label{tight} \begin{split} &   \lim_{n \to \infty}\frac{ \int_{\text{Mat}(\N, \C)} f(\mathfrak{r}^{(n)} (z))  \Psi_{g} (\mathfrak{r}^{(n)} (z) )  d \mu^{(s)} (z)}{\int_{\text{Mat}(\N, \C)}  \Psi_{g} (\mathfrak{r}^{(n)} (z) )  d \mu^{(s)} (z) }  \\  = &  \frac{ \int_{\text{Mat}(\N, \C)} f(\mathfrak{r}^{(\infty)} (z))  \Psi_{g} (\mathfrak{r}^{(\infty)} (z) )  d \mu^{(s)} (z)}{\int_{\text{Mat}(\N, \C)}  \Psi_{g} (\mathfrak{r}^{(\infty)} (z) )  d \mu^{(s)} (z) }. \end{split}   \end{align}  Moreover, by Lemma 1.14 in \cite{Bufetov_inf-deter}, there exists a positive bounded continuous function $f$ such that  \begin{align*}   \lim_{n \to \infty} \int_{\text{Mat}(\N, \C)} f(\mathfrak{r}^{(n)} (z))    d \mu^{(s)} (z) =   \int_{\text{Mat}(\N, \C)} f(\mathfrak{r}^{(\infty)} (z))   d \mu^{(s)} (z).   \end{align*} Again by Lemma \ref{DC-lem}, we have   \begin{align}\label{fg} \begin{split} &   \lim_{n \to \infty} \int_{\text{Mat}(\N, \C)} f(\mathfrak{r}^{(n)} (z))  \Psi_{g} (\mathfrak{r}^{(n)} (z) )  d \mu^{(s)} (z) \\  = &  \int_{\text{Mat}(\N, \C)} f(\mathfrak{r}^{(\infty)} (z))  \Psi_{g} (\mathfrak{r}^{(\infty)} (z) )  d \mu^{(s)} (z).\end{split}  \end{align}  Finally,  \eqref{wanted} follows from \eqref{tight} and \eqref{fg}, as desired.
\end{proof}

\begin{rem}
Note that \begin{align*}  \frac{n^2 x - 1}{ n^2 x + 1} = 1 - \frac{4/x}{2n^2 + 2/x} \sim 1 - \frac{4/x}{2n^2}.  \end{align*} Thus under change of variable $y = 4/x$, in the sequel, we only consider the scaling regimes of type $$x = 1 - \frac{z}{2n^2}. $$
\end{rem}

\section{Scaling Limits of Christoffel-Uvarov deformations of Jacobi Orthogonal Polynomial Ensembles.}

In this section, we will calculate explicitly the kernels for the determinatal probabilities  $\mathbb{P}_g^{(s)}$ given in Proposition \ref{abstract-result}. For avoiding extra notation, we mention here that in the sequel, in case I the $s$ corresponds to $s - 2m$ in Proposition \ref{abstract-result}, in cases II and III, it corresponds to $s + m $ in Propostion \ref{abstract-result}. For the case III, we give the result only for $ m = 2$.

 Observe that in the new coordinate $x= \rho (y)$, the kernel $K(x_1, x_2)$ for a locally trace class operator on $L^2(\R_+)$ reduces to $$\sqrt{\rho'(y_1) \rho'(y_2)} K(\rho(y_1), \rho(y_2)).$$

\subsection{Explicit Kernels for Scaling Limit: Case I}
Let $s > -1$. Consider a sequence $\xi^{(n)} = (\xi_1^{(n)}, \cdots, \xi_m^{(n)})$ of $m$-tuples of distinct real numbers in $(-1, 1)$. Let $w_s^{[\xi^{(n)}]}$ be the weight on $(-1, 1)$ given by   $$w_s^{[\xi^{(n)}]}(t)  = \prod_{i =1}^m (\xi_i^{(n)} - t)^2 \cdot w_s(t) =  \prod_{i =1}^m (\xi_i^{(n)} - t)^2 \cdot ( 1 - t)^s .$$

 Let $K_n^{[s, \xi^{(n)}]}(x_1, x_2)$ denote the $n$-th Christoffel-Darboux kernel for the weight $w_s^{[\xi^{(n)}]}$. The aim of this section is to establish the scaling limit of $K_n^{[s, \xi^{(n)}]}(x_1, x_2)$ in the following regime: \begin{align}\label{regime-case1}   \begin{split} \xi^{(n)}_i = 1 - \frac{w_i}{2 n^2}, 1 \le i \le m, &  \text{ $w_i > 0$ are all distinct;}  \\    x_i^{(n)}  = 1 - \frac{z_i}{2n^2}, & \quad z_i > 0, i = 1, 2. \end{split}
\end{align}

\subsubsection{Explicit formulae for orthogonal polynomials and Christoffel-Darboux kernels. }   

Let $ ( \pi_j^{[s, \xi^{(n)}]})_{j \ge 0} $ denote the system of  monic orthogonal polynomials associated with the weights $w_s^{[\xi^{(n)}]}$. To simplify notation, if there is no confusion, we denote   $  \pi_j^{[s, \xi^{(n)}]}$ by $\pi_j^{(n)}$.

 The monic polynomials $\pi_j^{(n)}$'s are given by the Christoffel formula (\cite[Thm 2.5.]{Szego-OP}):  \begin{align*}  \pi_j^{(n)}(t) = \frac{1}{\prod_{i = 1}^m (\xi^{(n)}_i-t)^2} \cdot \frac{D_j^{(n)} (t) }{ k_{j+2m}^{(s)} \cdot \delta_j^{(n)}},\end{align*} where

$$ D_j^{(n)} (t)  =   \left | \begin{array}{cccc} P_j^{(s)} (\xi^{(n)}_1) & P_{j+1}^{(s)}(\xi_1^{(n)}) & \cdots &  P_{j+2 m}^{(s)}(\xi^{(n)}_1) \vspace{3mm}

\\ \vdots & \vdots &  & \vdots

\\ P_j^{(s)} (\xi^{(n)}_m) & P_{j+1}^{(s)}(\xi_m^{(n)}) & \cdots &  P_{j+2 m}^{(s)}(\xi^{(n)}_m) \vspace{3mm}

\\ \dot{ P}_j^{(s)} (\xi^{(n)}_1) & \dot{P}_{j+1}^{(s)}(\xi_1^{(n)}) & \cdots &  \dot{P}_{j+2 m}^{(s)}(\xi^{(n)}_1) \vspace{3mm}

\\  \vdots & \vdots &  & \vdots

\\ \dot{P}_j^{(s)} (\xi^{(n)}_m) & \dot{P}_{j+1}^{(s)}(\xi_m^{(n)}) & \cdots &  \dot{P}_{j+2 m}^{(s)}(\xi^{(n)}_m) \vspace{3mm}

\\ P_j^{(s)} (t) & P_{j+1}^{(s)}(t) & \cdots &  P_{j+2 m}^{(s)}(t) \end{array} \right |;$$ and
$$\delta_j^{(n)}  =   \left | \begin{array}{cccc} P_j^{(s)} (\xi^{(n)}_1) & P_{j+1}^{(s)}(\xi_1^{(n)}) & \cdots &  P_{j+2 m-1}^{(s)}(\xi^{(n)}_1) \vspace{3mm}

\\ \vdots & \vdots &  & \vdots

\\ P_j^{(s)} (\xi^{(n)}_m) & P_{j+1}^{(s)}(\xi_m^{(n)}) & \cdots &  P_{j+2 m-1}^{(s)}(\xi^{(n)}_m) \vspace{3mm}

\\ \dot{ P}_j^{(s)} (\xi^{(n)}_1) & \dot{P}_{j+1}^{(s)}(\xi_1^{(n)}) & \cdots &  \dot{P}_{j+2 m -1}^{(s)}(\xi^{(n)}_1) \vspace{3mm}

 \\  \vdots & \vdots &  & \vdots

\\ \dot{P}_j^{(s)} (\xi^{(n)}_m) & \dot{P}_{j+1}^{(s)}(\xi_m^{(n)}) & \cdots &  \dot{P}_{j+2 m - 1}^{(s)}(\xi^{(n)}_m)  \end{array} \right |. $$

\begin{defn} Let $  h_j^{[s, \xi^{(n)}]} = \int_{-1}^1 \Big\{ \pi_j^{(n)} (t) \Big\}^2 w_s^{[\xi^{(n)}]}(t)dt  .$
\end{defn}

\begin{prop}
For any $j \ge 0$, we have
$$h_{j}^{[s, \xi^{(n)}]} = \frac{h_{j}^{(s)}}{k_{j}^{(s)} k_{j+2m}^{(s)}}\cdot \frac{\delta_{j+1}^{(n)}}{\delta_{j}^{(n)}}.$$
\end{prop}

\begin{proof}
By orthogonality, for any $\ell \ge 1$, we have $$\int_{-1}^1 P^{(s)}_{j+u}(t) \pi_j^{(n)}(t) w_s(t) dt = 0.$$ Note that $$D_j^{(n)}  = \delta_{j+1}^{(n)} P_j^{(s)}(t) +  \text{ linear combination of $ P_{j + 1}^{(s)}, \cdots, P_{j+2m}^{(s)}$.} $$ Hence
\begin{align*}  h_{j}^{[s, \xi^{(n)}]}  & =   \frac{1}{k_{j+ 2m}^{(s)}  \delta_{j}^{(n)}} \int D_{j}^{(n)} (t) \pi_j^{(n)} (t) w_s(t) dt   \\ & =   \frac{1}{k_{j + 2m}^{(s)} \delta_{j}^{(n)}} \int \delta_{j+1}^{(n)}  P_{j}^{(s)}  (t) \pi_j^{(n)} (t) w_s(t) dt  \\ & =  \frac{\delta_{j+1}^{(n)} }{k_{j + 2m}^{(s)}  \delta_{j}^{(n)}} \int  \Big\{P_{j}^{(s)}  (t) \Big\}^2 \frac{1}{k_j^{(s)}}w_s(t) dt  \\ & =  \frac{h_{j}^{(s)}}{k_{j}^{(s)} k_{j+2m}^{(s)}}\cdot \frac{\delta_{j+1}^{(n)}}{\delta_{j}^{(n)}}.\end{align*}
\end{proof}

By the Christoffel-Darboux formula (cf. \cite[Thm 3.2.2]{Szego-OP}), we have: \begin{align*}    & K_n^{[s, \xi^{(n)}]} (x_1^{(n)}, x_2^{(n)})   =    \sqrt{w_s^{[\xi^{(n)}]} (x_1^{(n)}) w_s^{[\xi^{(n)}]} (x_2^{(n)}) }\cdot \sum_{ j = 0}^{n-1} \frac{\pi_j^{(n)} (x_1^{(n)}) \cdot \pi_j^{(n)}(x_2^{(n)}) }{h_j^{[s, \xi^{(n)}]}}  \\ & =   \frac{\sqrt{w_s^{[\xi^{(n)}]} (x_1^{(n)}) w_s^{[\xi^{(n)}]} (x_2^{(n)}) } }{h_{n-1}^{[s, \xi^{(n)}]}}  \cdot \frac{\pi_n^{(n)}(x_1^{(n)}) \cdot \pi_{n-1}^{(n)} (x_2^{(n)}) - \pi_n^{(n)}(x_2^{(n)}) \cdot \pi_{n-1}^{(n)} (x_1^{(n)})}{ x_1^{(n)} - x_2^{(n)}}. \end{align*}

After change of variables $x_i^{(n)} = 1 - \frac{z_i }{2n^2}, z_i  \in [ 0, 4n^2], i = 1, 2$, and let $\xi^{(n)}$ takes the form as in the regime \eqref{regime-case1}, these kernels can be written as: \begin{align}\label{C-D-good} & \widetilde{K}_n^{[s, \xi^{(n)}]} (z_1, z_2)  =   \frac{1}{2 n^2} K_n^{[s, \xi^{(n)}]} \Big( 1 - \frac{z_1}{2 n^2}, 1 - \frac{z_2}{2 n^2}\Big) \\ & = \frac{(z_1 z_2)^{\frac{s}{2}}}{\left| \prod_{i=1}^m (z_1 - w_i)  (z_2 - w_i )\right|} \cdot S_n(z_1, z_2), \nonumber  \end{align} where \begin{align}\label{S-good-1} S_n(z_1, z_2) =  (2n^2)^{2m -s -1} \sum_{j = 0}^{n-1} \frac{D_j^{(n)} (1- \frac{z_1}{2n^2} ) D_j^{(n)} (1- \frac{z_2}{2n^2}   )}{ \frac{h_j^{(s)} k_{j+2m}^{(s)} }{ k_j^{(s)} } \delta_j^{(n)} \delta_{j+1}^{(n)} }, \end{align} or equivalently \begin{align}\label{S-good-2} &  \quad S_n(z_1, z_2)     =  \frac{(2n^2)^{2m-s}}{  \frac{h_{n-1}^{(s)} k_{n+2m}^{(s)} }{ k_{n-1}^{(s)} } \Big[ \delta_n^{(n)} \Big]^2 } \times  \\ & \times \frac{  D_n^{(n)}(1- \frac{z_1}{2n^2} ) \cdot D_{n-1}^{(n)} ( 1- \frac{z_2}{2n^2}  ) - D_n^{(n)}(1- \frac{z_2}{2n^2} ) \cdot D_{n-1}^{(n)} ( 1- \frac{z_1}{2n^2}  ) }{z_2 - z_1 }. \nonumber  \end{align}

\subsubsection{Scaling limits.} To obtain the scaling limit of the Christoffel-Darboux kernels $\widetilde{K}_n^{[s, \xi^{(n)}]} (z_1, z_2)$, we shall investigate the asymptotics of the formulae \eqref{S-good-1} or \eqref{S-good-2}. These two representations \eqref{S-good-1} and \eqref{S-good-2} will yield different representations of the scaling limit: an integrable form and an integral form.

The following formulae are well-known (\cite[p.63, p.68]{Szego-OP}): \begin{align}\label{leading-norm} k_j^{(s)} = \frac{1}{2^j \cdot j !} \frac{\Gamma(2 j+s + 1)}{\Gamma( j + s + 1)}, \quad h_j^{(s)} = \frac{2^{s+1}}{2 j + s +1}. \end{align} The following lemma will be used frequently in the sequel.
\begin{lem}
Let  $ p\in \Z$, then $$\lim_{n \to \infty} \frac{k_{\kappa_n + p}^{(s)}}{ k_{\kappa_n}^{(s)}}  = 2^p.$$
\end{lem}
\begin{proof}
It is an easy consequence of \eqref{leading-norm} and the Stirling's approximation formula for Gamma functions,  here we can also use the following convenient formula:  $$ \text{ for all } a \in \R, \quad\lim_{n \to \infty} \frac{\Gamma( n + a)}{n^a \Gamma(n)} = 1 . $$
\end{proof}

\begin{prop}\label{C}
If  the sequence $\xi^{(n)}$ satisfies  \eqref{regime-case1} ,  then $$ \lim_{n \to \infty} n^{2m^2 -2sm-3m} \delta_{\kappa_n}^{(n)} = \frac{2^{2ms}}{(w_1 \cdots w_m)^{1+s}} C_I^{(s, w)}(\kappa), $$ where $$C_I^{(s,w)}(\kappa) =W(J_{s, w_1}, \cdots, J_{s, w_m}, \widetilde{J}_{s+1, w_1}, \cdots \widetilde{J}_{s+1, w_m})(\kappa) $$ and $$J_{s, w_i} (\kappa) : = J_s(\kappa \sqrt{w_i}), \quad \widetilde{J}_{s+1, w_i} (\kappa)= \widetilde{J}_{s+1} (\kappa \sqrt{w_i}).$$ Moreover, the convergence is uniform as long as $\kappa$ is in a compact subset of $(0, \infty)$.








\end{prop}

\begin{proof}
To simplify notation, we denote $\Delta_{P, n}^{(s; \, \ell)}$ by $\Delta_{P, n}^{[\ell]}$ in this proof. By the multi-linearity of the determinant on columns,  we have \begin{eqnarray*} \delta_{\kappa_n}^{(n)} =   \left | \begin{array}{cccc} \Delta_{P, \kappa_n}^{[0]} (\xi^{(n)}_1) & \Delta_{P, \kappa_n}^{[1]} (\xi^{(n)}_1)  & \cdots &  \Delta_{P, \kappa_n}^{[2m-1]} (\xi^{(n)}_1) \vspace{3mm}

\\ \vdots & \vdots &  & \vdots

\\  \Delta_{P, \kappa_n}^{[0]} (\xi^{(n)}_m) & \Delta_{P, \kappa_n}^{[1]} (\xi^{(n)}_m)  & \cdots &  \Delta_{P, \kappa_n}^{[2m-1]} (\xi^{(n)}_m) \vspace{3mm}

\\ \dot{\Delta}_{P, \kappa_n}^{[0]} (\xi^{(n)}_1) & \dot{\Delta}_{P, \kappa_n}^{[1]} (\xi^{(n)}_1)  & \cdots &  \dot{\Delta}_{P, \kappa_n}^{[2m-1]} (\xi^{(n)}_1) \vspace{3mm}

\\ \vdots & \vdots &  & \vdots

\\  \dot{\Delta}_{P, \kappa_n}^{[0]} (\xi^{(n)}_m) & \dot{\Delta}_{P, \kappa_n}^{[1]} (\xi^{(n)}_m)  & \cdots &  \dot{\Delta}_{P, \kappa_n}^{[2m-1]} (\xi^{(n)}_m)   \end{array} \right | .\end{eqnarray*} Multiplying the matrix used in the above formula on right by the diagonal matrix $\text{diag}( n^{-s}, n^{1-s}, \cdots, n^{2m-1-s})$ and on left by the diagonal matrix $\text{diag}(\underbrace{1, \cdots, 1}_{\text{$m$ terms }}, \underbrace{n^{-2}, \cdots, n^{-2}}_{\text{$m$ terms} })$ and taking determinant, we obtain that $n^{2m^2 -2sm-3m} \delta_{\kappa_n}^{(n)}$ equals to the following determinant $$ \left | \begin{array}{cccc} n^{-s}\Delta_{P, \kappa_n}^{[0]} (\xi^{(n)}_1) & n^{1-s} \Delta_{P, \kappa_n}^{[1]} (\xi^{(n)}_1)  & \cdots & n^{2m-1-s} \Delta_{P, \kappa_n}^{[2m-1]} (\xi^{(n)}_1) \vspace{3mm}  \\ \vdots & \vdots &  & \vdots \\  n^{-s} \Delta_{P, \kappa_n}^{[0]} (\xi^{(n)}_m) & n^{1-s} \Delta_{P, \kappa_n}^{[1]} (\xi^{(n)}_m)  & \cdots &  n^{2m-1-s}\Delta_{P, \kappa_n}^{[2m-1]} (\xi^{(n)}_m) \vspace{3mm} \\ n^{-2-s}\dot{\Delta}_{P, \kappa_n}^{[0]} (\xi^{(n)}_1) & n^{-1-s}\dot{\Delta}_{P, \kappa_n}^{[1]} (\xi^{(n)}_1)  & \cdots & n^{2m-3-s} \dot{\Delta}_{P, \kappa_n}^{[2m-1]} (\xi^{(n)}_1) \vspace{3mm}  \\ \vdots & \vdots &  & \vdots \\  n^{-2-s}\dot{\Delta}_{P, \kappa_n}^{[0]} (\xi^{(n)}_m) & n^{-1-s}\dot{\Delta}_{P, \kappa_n}^{[1]} (\xi^{(n)}_m)  & \cdots &  n^{2m-3-s}\dot{\Delta}_{P, \kappa_n}^{[2m-1]} (\xi^{(n)}_m)   \end{array} \right |.$$ Applying Propositions \ref{jacobi-asymp} and \ref{der-asymp}, we obtain the desired formula. The last statement follows from the uniform convergences in Propositions \ref{jacobi-asymp}, \ref{der-asymp}.
\end{proof}

\begin{prop}\label{A}
If the sequences $x_i^{(n)}$ and  $\xi^{(n)}$ satisfy \eqref{regime-case1} ,  then \begin{align*}  \lim_{n \to \infty}   n^{2m^2 - m -2ms-s}  D^{(n)}_{\kappa_n}(x_i^{(n)}) =  \frac{2^{2ms +s}}{(w_1 \cdots w_m)^{1+s}} z_i^{-\frac{s}{2}}\cdot  A_I^{(s, w)} (\kappa, z_i), \end{align*} where $$A_I^{(s, w)}(\kappa, z_i) = W\Big(J_{s, w_1}, \cdots, J_{s, w_m}, \widetilde{J}_{s+1, w_1}, \cdots, \widetilde{J}_{s+1, w_m}, J_{s, z_i}\Big)(\kappa).$$
Moreover, the convergence is uniform as long as $\kappa$ is in a compact subset of $(0, \infty)$.








\end{prop}

\begin{proof}
The proof is similar to that of Proposition \ref{C}.
\end{proof}

\begin{defn} Define the column vector function $\boldsymbol{\theta}_j^{(n)}(t)$ by  $$ \boldsymbol{\theta}_j^{(n)}(t) =  \Big( P_j^{(s)} (\xi^{(n)}_1) , \cdots,  P_j^{(s)} (\xi^{(n)}_m) ,  \dot{ P}_j^{(s)} (\xi^{(n)}_1) , \cdots,
 \dot{P}_j^{(s)} (\xi^{(n)}_m), P_j^{(s)} (t) \Big)^T.$$
 \end{defn}

 \begin{prop}\label{B}
If the sequences $x_i^{(n)}$ and  $\xi^{(n)}$ satisfy \eqref{regime-case1}, then \begin{align*}  & \lim_{n \to \infty} n^{1+ 2m^2 - m -2ms-s} \left| \boldsymbol{\theta}_{\kappa_n}^{(n)}(x_i^{(n)}) \,  \cdots  \, \boldsymbol{\theta}_{\kappa_n +2m-1}^{(n)} (x_i^{(n)} ) \quad \boldsymbol{\theta}_{\kappa_n+2m}^{(n)} (x_i^{(n)})  - \boldsymbol{\theta}_{\kappa_n-1}^{(n)} (x_i^{(n)}) \right| \\ & =   \frac{2^{2ms +s}}{(w_1 \cdots w_m)^{1+s}} z_i^{-\frac{s}{2}}\cdot  B_I^{(s, w)}(\kappa, z_i), \end{align*} where  $B_I^{(s, w)}(\kappa, z_i)  = \left | \begin{array}{ccccc}\boldsymbol{\eta}_{s, z_i}(\kappa) & \boldsymbol{\eta}_{s, z_i}'(\kappa) & \cdots & \boldsymbol{\eta}_{s, z_i}^{(2m-1)}(\kappa) & \boldsymbol{\eta}_{s, z_i}^{(2m+1)}(\kappa) \end{array}  \right|$ and $\boldsymbol{\eta}_{s, z_i}(\kappa)$ is the column vector $$\Big(J_s(\kappa \sqrt{w_1}), \cdots, J_s(\kappa \sqrt{w_m} ), \widetilde{J}_{s+1}(\kappa \sqrt{w_1}), \cdots, \widetilde{J}_{s+1}(\kappa \sqrt{w_1}), J_s(\kappa \sqrt{z_i})\Big)^T.$$









\end{prop}

\begin{proof}
The proof is similar to that of Proposition \ref{C}, we emphasize that in the proof we used  the elementary fact \begin{align*}  \Delta_{P, \kappa_n-1}^{[2m+1]}   = &  P_{\kappa_n+2m}^{(s)} +(-1)^{2m+1} P_{\kappa_n-1}^{(s)}   \\  +  &  \text{ linear combination of $ P_{\kappa_n}^{(s)}, P_{\kappa_n+1}^{(s)}, \cdots, P_{\kappa_n + 2m - 1}^{(s)}$}. \end{align*}
\end{proof}

\begin{rem}\label{A-B-relation}
By the property of determinant, it is easy to  see that $$\frac{\partial}{\partial \kappa} A_I^{(s, w)} (\kappa, z_i)  = B_I^{(s, w)} (\kappa, z_i). $$
\end{rem}

\begin{thm}\label{thm-case1}
In the regime \eqref{regime-case1} ,   we have 
\begin{align*} &  \lim_{n \to \infty} \widetilde{K}_n^{[s, \xi^{(n)}]} (z_1, z_2)  \\ = &   \frac{A_I^{(s,w)}(1, z_1) B_I^{(s,w)}(1, z_2) - A_I^{(s,w)}(1, z_2) B_I^{(s,w)}(1, z_1) }{  2 \Big| \prod_{i= 1}^m (z_1 - w_i) (z_2 - w_i )\Big|\cdot \big[C_I^{(s, w)}(1) \big]^2 \cdot (z_1-z_2)}. \end{align*}
We denote this kernel by  $\mathscr{K}_\infty^{[s, \xi]} (z_1, z_2)$.
\end{thm}

\begin{proof}
It is easy to see that \begin{align*}  D_{n-1}^{(n)} (x_i^{(n)}) =   \left| \boldsymbol{\theta}_{n}^{(n)}(x_i^{(n)})  \cdots \boldsymbol{\theta}_{n+2m-1}^{(n)} (x_i^{(n)}) \quad \boldsymbol{\theta}_{n-1}^{(n)} (x_i^{(n)}) \right| ,  \end{align*}  hence $D_n^{(n)} (x_i^{(n)}) - D_{n-1}^{(n)} (x_i^{(n)}) $ equals to   $$  \left| \boldsymbol{\theta}_{n}^{(n)}(x_i^{(n)})  \cdots \boldsymbol{\theta}_{n+2m-1}^{(n)} (x_i^{(n)}) \quad \boldsymbol{\theta}_{n +2m}^{(n)} (x_i^{(n)}) -  \boldsymbol{\theta}_{n-1}^{(n)} (x_i^{(n)}) \right|.$$ Note that \begin{align*} & D_n^{(n)} (x_1^{(n)}) D_{n-1}^{(n)} (x_2^{(n)}) - D_n^{(n)} (x_2^{(n)}) D_{n-1}^{(n)} (x_1^{(n)}) \\ = & 
  D_n^{(n)}(x_2^{(n)} ) \Big[ D_n^{(n)} (x_1^{(n)}) - D_{n-1}^{(n)} (x_1^{(n)})\Big]  \\ & - D_n^{(n)}(x_1^{(n)} ) \Big[ D_n^{(n)} (x_2^{(n)}) - D_{n-1}^{(n)} (x_2^{(n)})\Big]. \end{align*} Now applying Propositions \ref{A} and \ref{B}, we obtain that  \begin{align*}   & \lim_{n \to \infty} n^{1 + 4m^2 - 2m - 4ms -2s } \Big[ D_n^{(n)} (x_1^{(n)}) D_{n-1}^{(n)} (x_2^{(n)}) - D_n^{(n)} (x_2^{(n)}) D_{n-1}^{(n)} (x_1^{(n)})\Big] \\  &=    \frac{2^{4ms +2s} (z_1z_2)^{-\frac{s}{2}} }{(w_1 \cdots w_m)^{2+2s}}  \Big( A_I^{(s, w)} (1, z_2) B_I^{(s,w)}(1, z_1) - A_I^{(s, w)} (1, z_1) B_I^{(s,w)}(1, z_2) \Big).   \end{align*} Combining with Proposition \ref{C}, we deduce that \begin{align*} &  \lim_{n \to \infty} S_n(z_1, z_2)  \\ = &      (z_1z_2)^{- \frac{s}{2}} \cdot \frac{A_I^{(s,w)}(1, z_1) B_I^{(s,w)}(1, z_2) - A_I^{(s,w)}(1, z_2) B_I^{(s,w)}(1, z_1) }{ 2 \big[C_I^{(s, w)} (1) \big]^2 (z_1-z_2)}. \end{align*} Substituting the above formula in \eqref{C-D-good}, we get the desired result.
\end{proof}

\begin{thm}\label{integral-form-1}
The kernel $ \mathscr{K}_\infty^{(s, \xi)}(z_1, z_2)$  has the following integral form: \begin{align*}  & \mathscr{K}_\infty^{(s, \xi)}(z_1, z_2)   \\ = &  \frac{1}{2 \Big| \prod_{i= 1}^m (z_1 - w_i) (z_2 - w_i )\Big|}  \int_0^1   \frac{ A_I^{(s, w)}(t, z_1)   A_I^{(s, w)}(t, z_2)  }{ \big[C_I^{(s, w)} (t) \big]^2 }  t dt .\end{align*}
\end{thm}

\begin{proof}
Let us fix $ z_1, z_2 > 0$. For any $\varepsilon > 0$, we can divide the sum in  \eqref{S-good-1} into two parts:\begin{align*} S_n(z_1, z_2)  & =  \underbrace{(2n^2)^{2m -s -1} \sum_{j = 0}^{\lfloor n \varepsilon \rfloor -1 }   \cdots}_{ = : \, I_n(\varepsilon)} \, +  \, \underbrace{(2n^2)^{2m -s -1} \sum_{j = \lfloor n \varepsilon \rfloor }^{n-1} \cdots}_{ = : \, II_n(\varepsilon)}.\end{align*}
The second term $II_n(\varepsilon)$  can be written as an integral: \begin{align*}  II_n(\varepsilon)  & =  \int_{\Big[\frac{\lfloor n \varepsilon \rfloor }{n}, 1\Big)}   \underbrace{(2n^2)^{2m -s -1} \frac{D_{\lfloor n t \rfloor}^{(n)} (x_1^{(n)} ) D_{\lfloor n t \rfloor}^{(n)} (x_2^{(n)}   )}{ \frac{h_{\lfloor n t \rfloor}^{(s)} k_{\lfloor n t \rfloor+2m}^{(s)} }{ k_{\lfloor n t \rfloor}^{(s)} } \delta_{\lfloor n t \rfloor}^{(n)} \delta_{\lfloor n t \rfloor+1}^{(n)} }    \cdot n}_{ = : \, T_n  (t) } \quad dt  .\end{align*} By Propositions \ref{C} and \ref{A}, we have the uniform convergence for $t \in [\varepsilon, 1]$:  \begin{align*} & \lim_{n \to \infty} T_n(t)  =   \frac{(z_1z_2)^{- \frac{s}{2}}}{2 (C^{[s, w]} (t) )^2 } A_I^{(s, w)}(t, z_1)   A_I^{(s, w)}(t, z_2)  t ,\end{align*} hence as $n \to \infty$, $II_n(\varepsilon)$ tends to $$II_\infty(\varepsilon) = \int_\varepsilon^1   \frac{(z_1z_2)^{- \frac{s}{2}}}{2 (C_I^{(s, w)} (t) )^2 } A_I^{(s, w)}(t, z_1)   A_I^{(s, w)}(t, z_2)  t dt . $$

For the first term $I_n(\varepsilon)$, we use Christoffel-Darboux formula to write it as \begin{align*}  \frac{(2n^2)^{2m-s}}{  \frac{h_{\lfloor n \varepsilon \rfloor-1}^{(s)} k_{\lfloor n \varepsilon \rfloor+2m}^{(s)} }{ k_{\lfloor n \varepsilon \rfloor-1}^{(s)} }  }  \frac{  D_{\lfloor n \varepsilon \rfloor}^{(n)}(x_1^{(n)} ) \cdot D_{\lfloor n \varepsilon \rfloor-1}^{(n)} ( x_2^{(n)}  ) - D_{\lfloor n \varepsilon \rfloor}^{(n)}(x_2^{(n)} ) \cdot D_{\lfloor n \varepsilon \rfloor-1}^{(n)} ( x_1^{(n)}  ) }{ \big[ \delta_{\lfloor n \varepsilon \rfloor}^{(n)} \big]^2  (z_2 - z_1) }.  \end{align*} By similar arguments as in the proof of Theorem \ref{thm-case1},  $$\lim_{n\to \infty} I_n(\varepsilon) = I_\infty(\varepsilon),$$ where $I_\infty(\varepsilon)$ is given by the formula \begin{align*}  & I_\infty(\varepsilon)  \\ = & \frac{(z_1z_2)^{-\frac{s}{2}}}{2} \cdot \frac{A_I^{(s,w)}(\varepsilon, z_1) B_I^{(s,w)}(\varepsilon, z_2) - A_I^{(s,w)}(\varepsilon, z_2) B_I^{(s,w)}(\varepsilon, z_1) }{ \big[C_I^{(s, w)}(\varepsilon) \big]^2 (z_1-z_2)} \cdot \varepsilon.\end{align*} Hence for any $\varepsilon > 0$, we have $$ \lim_{n \to \infty} S_n(z_1^{(n)}, z_2^{(n)}) = I_\infty(\varepsilon) + II_\infty(\varepsilon) .$$  The theorem is completely proved if we can establish $\lim_{ \varepsilon \to 0} I_\infty(\varepsilon) = 0. $ This is given by the following lemma.
\end{proof}

\begin{lem}\label{lem-case1}
For any $z_1, z_2 >0$, we have
$$\lim_{\varepsilon \to 0 + }  \frac{A_I^{(s,w)}(\varepsilon, z_1) B_I^{(s,w)}(\varepsilon, z_2) - A_I^{(s,w)}(\varepsilon, z_2) B_I^{(s,w)}(\varepsilon, z_1) }{ \big[C_I^{(s, w)}(\varepsilon) \big]^2} \cdot \varepsilon = 0.$$
\end{lem}

\begin{proof}
To simplify notation, let us denote $F_i = J_{s, w_i}$  and $G_i = \widetilde{J}_{s+1, w_i}$. We have $$C_I^{[s,w]}(\varepsilon) = W(F_1, \cdots, F_m, G_1, \cdots, G_m) (\varepsilon).$$ By \eqref{differential-relation-J}, we have $$G_i(\varepsilon)  = - \varepsilon F_i'(\varepsilon) + s F_i(\varepsilon). $$ If we denote $H_i(\varepsilon) = - \varepsilon F_i'(\varepsilon)$, then $$C^{[s, w]}(\varepsilon) = W(F_1,\cdots, F_m, H_1, \cdots, H_m)(\varepsilon).$$ We can write  $F_i(\varepsilon) = \sum_{\nu =0}^\infty \frac{a_\nu^{(i)}}{\nu !} \varepsilon^{2\nu + s}$ and $H_i(\varepsilon)= \sum_{\nu=0}^\infty \frac{b_\nu^{(i)}}{\nu ! } \varepsilon^{2 \nu + s}$,
with $$ a_\nu^{(i) } = \frac{(-1)^{\nu} (\sqrt{w_i})^{2\nu + s} }{2^{2\nu + s} \Gamma(\nu + s+1) } , \quad b_\nu^{(i)} =  - (2\nu +s) a_\nu^{(i)}. $$ Define entire functions:  $$f_i(x) =\sum_{\nu = 0}^\infty \frac{a_\nu^{(i)}}{\nu !} x^{\nu}, \quad h_i(x) = \sum_{\nu = 0}^\infty \frac{b_\nu^{(i)} }{\nu !}x^\nu.$$ Then $F_i(\varepsilon) = \varepsilon^s f_i(\varepsilon^2 )$ and $H_i(\varepsilon) = \varepsilon^s h_i(\varepsilon^2)$.   Using the  identity $$W(g f_1, \cdots, g f_n)(x) = g(x)^n \cdot W(f_1, \cdots, f_n)(x),$$ we obtain that $$C_I^{(s, w)} (\varepsilon) =\varepsilon^{2ms} \cdot W\Big( f_1(x^2), \cdots, f_m(x^2), h_1(x^2), \cdots, h_m(x^2) \Big)(\varepsilon). $$  An application of the following identity $$ \frac{d^{n}}{dx^n} \Big[ f(x^2) \Big] =  n ! \sum_{k = 0}^{\lfloor n/2\rfloor} \frac{(2x)^{n-2k}}{ k ! (n - 2k)!} f^{(n-k)} (x^2)$$ yields  $$C_I^{(s, w)}(\varepsilon) = 2^{m(2m-1)} \varepsilon^{2ms  + m(2m-1)} W(f_1, \cdots, f_m, h_1, \cdots, h_m)(\varepsilon^2). $$

We state the following simple auxiliary
\begin{lem}\label{sublem}
$W(f_1, \cdots, f_m, h_1, \cdots, h_m)(z)$ is an entire function and does not vanish at $z= 0$.
\end{lem}

Before proving Lemma \ref{sublem}, we derive  from it Lemma \ref{lem-case1}. Indeed, from Lemma \ref{sublem}  we have $$C_I^{(s, w)}(\varepsilon) \asymp \varepsilon^{2ms + m(2m-1)}  \text{ as } \varepsilon \to 0.$$ Similarly, $$ A_I^{(s, w)}(\varepsilon, z_i ) \asymp \varepsilon^{(2m + 1)s + m(2m + 1)}  \text{ as } \varepsilon \to 0. $$ By Remark \ref{A-B-relation}, we also have $$B_I^{(s, w)}(\varepsilon, z_i ) \asymp \varepsilon^{(2m + 1)s + m(2m + 1) - 1 }   \text{ as } \varepsilon \to 0. $$ Hence as $ \varepsilon \to 0, $ we have \begin{align*} \left |\frac{A_I^{(s,w)}(\varepsilon, z_1) B_I^{(s,w)}(\varepsilon, z_2) - A_I^{(s,w)}(\varepsilon, z_2) B_I^{(s,w)}(\varepsilon, z_1) }{ \Big[C_I^{(s, w)}(\varepsilon) \Big]^2} \cdot \varepsilon \right | \lesssim \varepsilon^{4m+2s }. \end{align*} Since we always have $4m + 2s  >  0$, Lemma \ref{lem-case1}  is proved.

\bigskip

Now we turn to the proof of the Lemma \ref{sublem}. By definition, we know that $f_i, h_i$ are entire functions, hence $W(f_1, \cdots, f_m, h_1, \cdots, h_m)$ is also entire.  It is easily to see that $W(f_1, \cdots, f_m, h_1, \cdots, h_m) (0)$  equals to  $$\left|\begin{array}{ccccc} a^{(1)}_0 & a^{(1)}_1 & a^{(1)}_2 & \cdots & a^{(1)}_{2m-1} \\ \vdots & \vdots & \vdots &  &\vdots \\ a^{(m)}_0 & a^{(m)}_1 & a^{(m)}_2 & \cdots & a^{(m)}_{2m-1}  \vspace{3mm}   \\ b^{(1)}_0 & b^{(1)}_1 & b^{(1)}_2 & \cdots & b^{(1)}_{2m-1} \\ \vdots & \vdots & \vdots & &\vdots \\ b^{(m)}_0 & b^{(m)}_1 & b^{(m)}_2 & \cdots & b^{(m)}_{2m-1}  \end{array} \right|, $$ which is in turn given by  a non-zero multiple of $  \det \mathscr{W},  $ where $$ \mathscr{W} = \left( \begin{array}{ccccc} 1 & w_1 & w_1^2 & \cdots &  w_1^{2m-1} \\ \vdots & \vdots & \vdots &  & \vdots \\ 1 & w_m & w_m^2 & \cdots & w_m^{2m-1}  \vspace{3mm} \\ 0 & 1 &  2 w_1 & \cdots &  (2m-1) w_1^{2m-2} \\ \vdots & \vdots & \vdots &  & \vdots \\ 0 & 1 & 2 w_m & \cdots & (2m-1) w_m^{2m-2}  \end{array}\right). $$   We claim that $\det \mathscr{W} \ne 0$. Indeed, let $\theta = (\theta_0, \theta_1, \cdots, \theta_{2m-1})^{T} $ be such that $\mathscr{W} \theta = 0$. In other words, we have \begin{align*}  \sum_{k=0}^{2m-1} \theta_k w_i^k = 0, \quad \sum_{k =0}^{2m-1} k \theta_kw_i^{k - 1} = 0, \text{ for } 1 \le i \le m .\end{align*} Let $\Theta$ be the polynomial given by $\Theta(x)  = \sum_{k =0}^{2m-1} \theta_k x^k,$ then the above equations imply that $w_1, \cdots, w_m$ are distinct roots of $\Theta$, each $w_i$ has multiplicity at least 2. Since $\deg \Theta \le 2m-1$, we must have $\Theta \equiv 0$ and hence $\theta=0$. This shows that $\mathscr{W}$ is invertible, hence has a non-zero determinant.

\end{proof}

\subsection{Explicit Kernels for Scaling Limit: Case II} Consider a sequence of $m$-tuples of distinct positive real numbers $r^{(n)} = (r_1^{(n)}, \cdots, r_m^{(n)}) $ and the modified weights $w_s^{(r^{(n)})}$ given as follows: $$w_s^{(r^{(n)})} (t) =\frac{w_s(t)}{\prod_{ i = 1}^m ( 1 + r_i^{(n)} - t)} = \frac{(1 - t)^s}{ \prod_{i = 1}^m ( 1 + r^{(n)}_i - t) }. $$ The $n$-th Christoffel-Darboux kernel associated with $w_s^{(r^{(n)})}$ is denoted by $ \Pi_n^{(n)} (x_1, x_2)$.

We will investigate the scaling limit of $\Pi_n^{(n)} (x_1^{(n)}, x_2^{(n)})$ in the regime: \begin{align}\label{case-II} \begin{split}  x_i^{(n)} = 1 - \frac{z_i }{2n^2}, \quad &  z_i > 0, i = 1, 2. \\  r_i^{(n)} = \frac{v_i}{2n^2},   1 \le i \le m, & \text{ and $v_i> 0$ are all distinct. }  \end{split} \end{align}

\subsubsection{Explicit formulae for orthogonal polynomials and Christoffel-Darboux kernels} The Christoffel-Uvarov formula implies that the following polynomials $q_j^{(n)}$ for $j \ge m$ are orthogonal with respect to  $ w_s^{( r^{(n)} )}$: $$q_j^{(n)} (t) =  \left| \begin{array}{ccc} Q_{j-m}^{(s)} ( 1 + r_1^{(n)} )  & \cdots  &  Q_{j}^{(s)}(1 + r_1^{(n)}) \\ \vdots & & \vdots \\Q_{j-m}^{(s)} ( 1 + r_m^{(n)})  & \cdots  &  Q_{j}^{(s)}(1 + r_m^{(n)})  \\ P_{j-m}^{(s)}(t) & \cdots & P_j^{(s)}(t) \end{array}  \right|. $$ For $0 \le j <  m$, we also denote by $q_j^{(n)}$ the $j$-th monic orthogonal polynomial, here we will not give its explicit formula.

Denote $$ d_j^{(n)}  = \left| \begin{array}{ccc} Q_{j-m}^{(s)} ( 1 + r_1^{(n)})  & \cdots  &  Q_{j-1}^{(s)}(1 + r_1^{(n)}) \\ \vdots & & \vdots \\Q_{j-m}^{(s)} ( 1 + r_m^{(n)})  & \cdots  &  Q_{j-1}^{(s)}(1 + r_m^{(n)}) \end{array}  \right|.$$ Denote by $k_j^{(s, r^{(n)})}$  the leading coefficient of  $q_j^{(n)}$. When $j \ge m$, it is given by $k_j^{(s, r^{(n)})} = d_j^{(n)} k_j^{(s)} $.


\begin{defn}
Define $h_j^{(s, r^{(n)})} = \int_{-1}^1 \Big\{ q_j^{(n)}(t)  \Big\}^2 w_s^{(r^{(n)})} (t) dt. $
\end{defn}

\begin{prop}
For any $j \ge m$, we have $$h_j^{(s, r^{(n)} )  } =  \frac{d_{j}^{(n)} d_{j+1}^{(n)}  k_j^{(s)}h_{j-m}^{(s)} } {k_{j-m}^{(s)}}. $$
\end{prop}

\begin{proof}
Let $j \ge m$, by the orthogonality, we have \begin{align*} & h_{j}^{(s,r^{(n)})}   \\ = &  \int q_j^{(n)}(t) d_{j}^{(n)} P_j^{(s)}(t) w_s^{(r^{(n)})}(t) dt  = d_{j}^{(n)} k_j^{(s)}  \int q_j^{(n)}(t) t^j w_s^{(r^{(n)})}(t) dt  \\  = & d_{j}^{(n)} k_j^{(s)}  \int q_j^{(n)}(t) (-1)^m \prod_{i = 1}^m (1+r_i^{(n)} - t)  t^{j-m} w_s^{(r^{(n)})}(t) dt   \\  = & (-1)^m d_{j}^{(n)} k_j^{(s)}  \int q_j^{(n)}(t) t^{j-m} w_s(t) dt \\  =  & (-1)^m d_{j}^{(n)} k_j^{(s)}  \int (-1)^{m + 2} d_{j+1}^{(n)} P_{j-m}^{(s)} (t) t^{j-m} w_s(t) dt  \\  = & \frac{d_{j}^{(n)} d_{j+1}^{(n)}  k_j^{(s)} } {k_{j-m}^{(s)}} \int \Big\{ P_{j-m}^{(s)}(t) \Big\}^2  w_s(t) dt  \\  = &  \frac{d_{j}^{(n)} d_{j+1}^{(n)}  k_j^{(s)}h_{j-m}^{(s)} } {k_{j-m}^{(s)}}  .\end{align*}
\end{proof}

By change of variables, $x_i^{(n)} = 1 - \frac{z_i}{2n^2} $, and let $r^{(n)}$ be as in the regime \eqref{case-II}, the Christoffel-Darboux kernels are given by the formula \begin{align}\label{CD-2} \begin{split} & \widetilde{\Pi}_n^{(n)} ( z_1, z_2)  = \frac{(z_1 z_2 )^{\frac{s}{2}} }{ \prod_{ i = 1}^m ( v_i+ z_1 )^{\frac{1}{2}} (v_i + z_2)^{\frac{1}{2}}}  \Sigma_n (z_1, z_2 ), \end{split} \end{align} where \begin{align}\label{sigma-1}  \Sigma_n(z_1, z_2)  = (2n^2)^{m-s-1}   \sum_{j = 0}^{n-1} \frac{q_j^{(n)} ( 1 - \frac{z_1}{2n^2} ) q_j^{(n)} (1 - \frac{z_2 }{2n^2} )  }{   \frac{d_{j}^{(n)} d_{j+1}^{(n)}  k_j^{(s)}h_{j-m}^{(s)} } {k_{j-m}^{(s)}}   }, \end{align} or equivalently \begin{align} \begin{split}  & \Sigma_n(z_1, z_2) \\   = &  \label{sigma-2} \frac{(2n^2)^{m-s}}{   \big[d_{n}^{(n)}\big]^2 \frac{h_{n-1 -m }^{(s)} k_{n}^{(s)} }{ k_{n-1-m}^{(s)}}   } \cdot \frac{
q_{n}^{(n)}(x_1^{(n)} ) q_{n-1}^{(n)}(x_2^{(n)} ) - q_{n}^{(n)}(x_2^{(n)}) q_{n-1}^{(n)}(x_1^{(n)} )
}{
z_2 - z_1} .  \end{split}\end{align}

\subsubsection{Scaling limits}  Now we investigate the scaling limits.

\begin{prop}\label{AC-matrix}
 In the regime \eqref{case-II} we have $$\lim_{n \to \infty} n^{\frac{m(m-1)}{2} - ms } d_{\kappa_n}^{(n)}  = 2^{ms} (v_1 \cdots v_m)^{- \frac{s}{2} } C_{II}^{(s, v)}(\kappa), $$

 $$\lim_{n \to \infty} n^{\frac{m(m+1)}{2} - (m+1)s } q_{\kappa_n}^{(n)} (x_i^{(n)})   = 2^{(m+1)s}  (v_1 \cdots v_m)^{- \frac{s}{2} } z_i^{-\frac{s}{2}} A_{II}^{(s, v)}(\kappa, z_i), $$ where $$C_{II}^{(s, v)} (\kappa) = W\Big(K_{s, v_1}, \cdots, K_{s, v_m} \Big)(\kappa),$$ $$A_{II}^{(s,v)}(\kappa, z) =W\Big(K_{s, v_1}, \cdots, K_{s, v_m}, J_{s, z}\Big ) (\kappa),  $$ and   $K_{s, v_i} (\kappa) = K_s(\kappa \sqrt{v_i})$,  $J_{s, z}(\kappa) = J_s(\kappa \sqrt{z})$.



\end{prop}

\begin{proof}
For $\ell \ge 1$, we have $$\Delta_{Q, n}^{(s, \ell)} = Q_{n + \ell}^{(s)} + (-1)^{\ell} Q_n^{(s)} + \text{linear combination of $Q_{n+1}^{(s)}, \cdots, Q_{n+\ell-1 }^{(s)}$}.$$ The same is true for $\Delta_{P, n}^{(s, \ell)}$ and with the same coefficients. Hence for $k_n \ge m$, we have  \begin{align*} d_{\kappa_n}^{(n)}   =   \left| \begin{array}{ccc} \Delta_{Q, \kappa_n-m}^{(s, 0)} ( 1 + r_1^{(n)})  & \cdots  &  \Delta_{Q, \kappa_n-m}^{(s, m-1)} ( 1 + r_1^{(n)}) \\ \vdots & & \vdots \\ \Delta_{Q, \kappa_n-m}^{(s, 0)} ( 1 + r_m^{(n)})  & \cdots  &  \Delta_{Q, \kappa_n-m}^{(s, m-1)} ( 1 + r_m^{(n)})  \end{array}  \right|;  \end{align*}

\begin{align*} q_{\kappa_n}^{(n)} (x_i^{(n)} )  =     \left| \begin{array}{ccc} \Delta_{Q, \kappa_n-m}^{(s, 0)} ( 1 + r_1^{(n)})  & \cdots  &  \Delta_{Q, \kappa_n-m}^{(s, m)} ( 1 + r_1^{(n)}) \\ \vdots & & \vdots \\ \Delta_{Q, \kappa_n-m}^{(s, 0)} ( 1 + r_m^{(n)})  & \cdots  &  \Delta_{Q, \kappa_n-m}^{(s, m)} ( 1 + r_m^{(n)})   \vspace{3mm} \\ \Delta_{P, \kappa_n-m}^{(s, 0)} ( x_i^{(n)})  & \cdots  &  \Delta_{P, \kappa_n-m}^{(s, m)} ( x_i^{(n)})   \end{array}  \right| . \end{align*}
The proposition is completely proved by applying the same arguments as in the proof of Proposition \ref{C} and by applying Propositions \ref{jacobi-asymp}, \ref{der-asymp}, \ref{asymp-Q} and \ref{asymp-diff-Q}.
\end{proof}

\begin{prop}\label{B-matrix}
In the regime \eqref{case-II}, we have \begin{align*}  & \lim_{n \to \infty} n^{\frac{m(m+1)}{2} - (m + 1)s + 1 } \Big[ q_{\kappa_n}^{(n )} (x_i^{(n)})  -  q_{\kappa_n-1}^{(n)} (x_i^{(n)})  \Big] \\ &  = 2^{(m +1)s} (v_1 \cdots v_m)^{-\frac{s}{2}} z_i^{-\frac{s}{2}}  B_{II}^{(s, v)}(\kappa, z_i),  \end{align*}  where \begin{align*} & B_{II}^{(s, v)}(\kappa, z) = \frac{\partial}{\partial \kappa} A_{II}^{(s, v)}(\kappa, z) \\ = & \left| \boldsymbol{\phi}_{s, z}(\kappa) , \boldsymbol{\phi}_{s, z}'(\kappa), \cdots, \boldsymbol{\phi}_{s, z}^{(m-1)}(\kappa), \boldsymbol{\phi}_{s, z}^{(m+1)}(\kappa)    \right|, \end{align*} and  $\boldsymbol{\phi}_{s, z}(\kappa)$ is   the column vector $\Big(K_s( \kappa \sqrt{v_1}), \cdots K_s( \kappa \sqrt{v_m}), J_s( \kappa \sqrt{z})\Big)^T$.
\end{prop}

\begin{proof}
To simplify notation, we show the proposition in the case $\kappa_n = n$, the proof in the general case is similar.  Define column vector $$ \beta_j^{(n)}(t) = \Big( Q_j^{(s)}(1 + r_1^{(n)}), \cdots, Q_j^{(s)}(1  + r_m^{(n)}), P_j^{(s)}(t) \Big)^T.$$ Then for $ i = 1, 2$, \begin{align*}  q_n^{(n)} (x_i^{(n)})   = & \left| \begin{array}{cccc} \beta_{n-m}^{(n)}(x_i^{(n)}) &\cdots & \beta_{n-1}^{(n)}(x_i^{(n)})  &  \beta_{n}^{(n)}(x_i^{(n)})   \end{array} \right|; \\ q_{n-1}^{(n)} (x_i^{(n)})  =&  \left| \begin{array}{cccc} \beta_{n-1-m}^{(n)}(x_i^{(n)}) & \cdots  & \beta_{n-2}^{(n)}(x_i^{(n)})  & \beta_{n-1}^{(n)}(x_i^{(n)})   \end{array} \right|  \\  = & (-1)^m  \left| \begin{array}{cccc} \beta_{n-m}^{(n)}(x_i^{(n)}) & \cdots & \beta_{n-1}^{(n)}(x_i^{(n)}) & \beta_{n-1-m}^{(n)}(x_i^{(n)})    \end{array} \right|. \end{align*} Hence \begin{align*}   & q_n^{(n)} (x_i^{(n)})  -  q_{n-1}^{(n)} (x_i^{(n)})  \\   =&   \left| \begin{array}{cccc} \beta_{n-m}^{(n)}(x_i^{(n)}) &\cdots & \beta_{n-1}^{(n)}(x_i^{(n)})  &  \beta_{n}^{(n)}(x_i^{(n)}) + (-1)^{m+1} \beta_{n-1-m}^{(n)}(x_i^{(n)})    \end{array} \right| \\ = & \left| \begin{array}{cccc}\Delta_{Q,n-m}^{(s, 0)} ( 1 + r_1^{(n)}) & \cdots & \Delta_{Q,n-m}^{(s, m-1)} ( 1 + r_1^{(n)}) & \Delta_{Q,n-1-m}^{(s, m+1)} ( 1 + r_1^{(n)}) \\\vdots &  & \vdots & \vdots \\\Delta_{Q,n-m}^{(s, 0)} ( 1 + r_m^{(n)}) & \cdots & \Delta_{Q,n-m}^{(s, m-1)} ( 1 + r_m^{(n)}) & \Delta_{Q,n-1-m}^{(s, m+1)} ( 1 + r_m^{(n)}) \vspace{3mm}\\ \Delta_{P,n-m}^{(s, 0)} ( x_i^{(n)}) & \cdots & \Delta_{P,n-m}^{(s; m-1)} ( x_i^{(n)}) & \Delta_{P,n-1-m}^{(s; m+1)} ( x_i^{(n)})\end{array} \right| .\end{align*}
We finish the proof by using Propositions \ref{jacobi-asymp}, \ref{der-asymp}, \ref{asymp-Q} and \ref{asymp-diff-Q}.
\end{proof}

Combining Propositions \ref{AC-matrix} and \ref{B-matrix}, we obtain
\begin{cor}\label{corollary}
In the regime \eqref{case-II}, we have \begin{align*}  & \lim_{n \to \infty} n^{m(m+1) - 2(m+1)s +1} \Big[  q_{\kappa_n}^{(n)}(x_1^{(n)}) q_{\kappa_n-1}^{(n)}(x_2^{(n)}) - q_{\kappa_n}^{(n)}(x_2^{(n)}) q_{\kappa_n-1}^{(n)}(x_1^{(n)}) \Big]  \\   & =  2^{2(m+1)s} (v_1 \cdots v_m)^{-s} z_1^{-\frac{s}{2}}z_2^{-\frac{s}{2}} \left| \begin{array}{cc}  A_{II}^{(s, v)}(\kappa, z_1)  &  - B_{II}^{(s,v)}(\kappa, z_1)\vspace{2mm}\\ A_{II}^{(s, v)}(\kappa, z_2)  &  - B_{II}^{(s,v)}(\kappa, z_2)   \end{array} \right|.\end{align*}
\end{cor}

 \begin{proof} We first write $q_{\kappa_n}^{(n)}(x_1^{(n)}) q_{\kappa_n-1}^{(n)}(x_2^{(n)}) - q_{\kappa_n}^{(n)}(x_2^{(n)}) q_{\kappa_n-1}^{(n)}(x_1^{(n)})$ as \begin{align*} \left| \begin{array}{cc} q_{\kappa_n}^{(n)}(x_1^{(n)})& q_{\kappa_n-1}^{(n)}(x_1^{(n)}) \vspace{3mm} \\ q_{\kappa_n}^{(n)}(x_2^{(n)})& q_{\kappa_n-1}^{(n)}(x_2^{(n)})   \end{array} \right| =  \left| \begin{array}{cc} q_{\kappa_n}^{(n)}(x_1^{(n)})& q_{\kappa_n-1}^{(n)}(x_1^{(n)}) - q_{\kappa_n}^{(n)}(x_1^{(n)})  \vspace{3mm} \\ q_{\kappa_n}^{(n)}(x_2^{(n)})& q_{\kappa_n-1}^{(n)}(x_2^{(n)}) - q_{\kappa_n}^{(n)}(x_2^{(n)})   \end{array} \right|. \end{align*} The corollary now follows from  Propositions \ref{AC-matrix} and \ref{B-matrix}.
\end{proof}

\begin{thm}\label{thm-case2-1}
In the regime \eqref{case-II}, we obtain the scaling limit
\begin{align*}  & \Pi_\infty^{(s, v)} (z_1, z_2)  : =  \lim_{n \to \infty} \widetilde{\Pi}_n^{(n)} (z_1, z_2) \\  =&  \frac{ A_{II}^{(s,v)}(1, z_1) B_{II}^{(s,v)}(1, z_2) -  A_{II}^{(s,v)}(1, z_2) B_{II}^{(s,v)}(1, z_1)}{ 2 \prod_{i=1}^m \sqrt{(v_i + z_1) ( v_i + z_2)} \cdot \big[  C_{II}^{(s,v)}(1)\big]^2 \cdot (z_1- z_2)}. \end{align*}
\end{thm}

\begin{proof}
By  \eqref{sigma-2} and Proposition \ref{AC-matrix}, Corollary \ref{corollary},  we have  \begin{align*} & \lim_{n \to \infty} \Sigma_n( z_1, z_2) \\ = &  \frac{(z_1z_2)^{-\frac{s}{2}}  \Big\{A_{II}^{(s,v)}(1, z_1) B_{II}^{(s,v)}(1, z_2) -  A_{II}^{(s,v)}(1, z_2) B_{II}^{(s,v)}(1, z_1)\Big\} }{2 \big[C_{II}^{(s,v)} (1)\big]^2 ( z_1 - z_2) }.\end{align*} Combining this with \eqref{CD-2}, we get the desired result.
\end{proof}

 \begin{prop}
 Let $s > m - 1, s \notin \N$. The kernel $\Pi_\infty^{(s, v)}(z_1, z_2)  $ has the following integral form: \begin{align*} &  \Pi_\infty^{(s, v)}(z_1, z_2) \\ = & \frac{1}{2 \prod_{i = 1}^m \sqrt{(v_i + z_1) (v_i + z_2)}}     \int_{0}^1  \frac{ A_{II}^{(s,v)}(\kappa, z_1)  \cdot   A_{II}^{(s,v)}(\kappa, z_2) }{\big[ C_{II}^{(s,v)}(\kappa)\big]^2} \kappa d \kappa.\end{align*}
\end{prop}

\begin{proof}
The proof is similar to that of Theorem \ref{integral-form-1}, a slight difference is, instead of using Lemma \ref{lem-case1}, we shall use the following Lemma \ref{lem-case2}.
\end{proof}

\begin{lem}\label{lem-case2}
Let $s > m - 1, s \notin \N$.  For any $z_1, z_2 > 0$, we have $$\lim_{ \varepsilon \to 0^{+}}   \frac{ A_{II}^{(s,v)}(\varepsilon, z_1) B_{II}^{(s,v)}(\varepsilon, z_2) -  A_{II}^{(s,v)}(\varepsilon, z_2) B_{II}^{(s,v)}(\varepsilon, z_1)}{  \big[  C_{II}^{(s,v)}(\varepsilon)\big]^2 } \cdot  \varepsilon = 0.$$
\end{lem}

 \begin{proof}
 Recall that $C_{II}^{(s, v)} (\varepsilon) = W\Big(K_{s, v_1}, \cdots, K_{s, v_m}(\varepsilon) \Big)(\varepsilon). $  By McDonald definition of $K_s(z)$, namely $$K_s(z) = \frac{\pi}{2} \frac{I_{-s} (z) - I_s(z)}{\sin s \pi}, \quad I_s(z) = \sum_{\nu = 0}^\infty \frac{(\frac{1}{2} z)^{2\nu + s }}{\nu ! \Gamma(\nu + s + 1)}, $$ we can write $$K_{s, v_i} (\varepsilon) = K_s(\varepsilon \sqrt{v_i} ) = \frac{\pi}{2 \sin(  s \pi)}\varepsilon^{- s}  \Big( \underbrace{  \sum_{ \nu = 0}^\infty \frac{\alpha^{(i) }_\nu}{\nu! } \varepsilon^{2\nu} - \varepsilon^{2  s}\sum_{\nu = 0}^\infty  \frac{\beta_{\nu}^{(i)}}{\nu ! } \varepsilon^{2 \nu }}_{: = \mathscr{A}_i (\varepsilon^2) } \Big),    $$ where $$ \alpha_\nu^{(i)}  = \frac{ (\frac{1}{2} \sqrt{v_i})^{2\nu - s}}{\Gamma(\nu - s +1)}, \quad \beta_\nu^{(i)}  = \frac{ (\frac{1}{2} \sqrt{v_i})^{2\nu + s}}{\Gamma(\nu +  s +1)}. $$ Thus \begin{align*} C_{II}^{(s,v)}(\varepsilon) & = \text{Const}_s  \times \varepsilon^{- m s} W\Big(\mathscr{A}_1(x^2), \cdots, \mathscr{A}_m(x^2)\Big) (\varepsilon) \\  & = \text{Const}_s \times \varepsilon^{-ms + \frac{m(m-1)}{2} } W\Big(\mathscr{A}_1, \cdots, \mathscr{A}_m \Big) (\varepsilon^2).   \end{align*} By the assumption $s > m -1$, we know that  $\mathscr{A}_i$ are all differentiable up to order at least $m-1$ on the neighbourhood of $0$, hence $ W(\mathscr{A}_1, \cdots, \mathscr{A}_m ) (\varepsilon)$ is continuous, and $$ W\Big(\mathscr{A}_1, \cdots, \mathscr{A}_m \Big) (0)  = \text{non-zero term} \times   \left| \begin{array}{cccc} 1 & v_1 & \cdots & v_1^{m-1} \\ \vdots & \vdots & & \vdots \\  1 & v_m & \cdots & v_m^{m-1} \end{array} \right| \ne 0 . $$ Hence $$ C_{II}^{(s, v)} (\varepsilon) \asymp \varepsilon^{-ms + \frac{m (m-1)}{2} }, \text{ as } \varepsilon \to 0. $$ In a similar way, we can show that $$A_{II}^{(s,v)} (\varepsilon, z_i) \asymp \varepsilon^{- (1+m) s + \frac{m(m+1)}{2} +  2[s - (m-1)]}   , \text{ as } \varepsilon \to 0.  $$ $$B_{II}^{(s,v)} (\varepsilon, z_i) \asymp \varepsilon^{- (1+m) s + \frac{m(m+1)}{2} +  2(s - m)}  , \text{ as } \varepsilon \to 0.$$ Hence $$\left|    \frac{ A_{II}^{(s,v)}(\varepsilon, z_1) B_{II}^{(s,v)}(\varepsilon, z_2) -  A_{II}^{(s,v)}(\varepsilon, z_2) B_{II}^{(s,v)}(\varepsilon, z_1)}{  \big[  C_{II}^{(s,v)}(\varepsilon)\big]^2 } \varepsilon \right| \lesssim  \varepsilon^{ 2s -2m +3}. $$ Since $2s - 2m + 3 > 0$, the lemma is completely proved.

 \end{proof}

{\flushleft \bf Remark.} Let us consider the case where $-1 < s < 0$ and $m = 1$. Let us denote \begin{align} \label{I-formula} \begin{split} & \mathscr{I}^{(s,v)} (\kappa, z_1, z_2) \\ & =       \frac{  \left| \begin{array}{cc} K_s(\kappa \sqrt{v}) & \sqrt{v} K_s'(\kappa \sqrt{v}) \\ J_s(\kappa \sqrt{z_1}) & \sqrt{z_1} J_s'(\kappa \sqrt{z_1}) \end{array}\right|   \left| \begin{array}{cc} K_s(\kappa \sqrt{v}) & \sqrt{v} K_s'(\kappa \sqrt{v}) \\ J_s(\kappa \sqrt{z_2}) & \sqrt{z_2} J_s'(\kappa \sqrt{z_2}) \end{array}\right| }{   2  \sqrt{(v + z_1) (v + z_2)} \cdot  \big[ K_s(\kappa \sqrt{v})\big]^2} \cdot \kappa . \end{split} \end{align}  For any $\varepsilon$, we divide the following sum into two parts: \begin{align*}  & \Pi_n^{(n)} (z_1, z_2)  = \frac{1}{2n^2}  \sum_{j = 0}^{n-1} \frac{ q_j(x_1^{(n)}) q_j(x_2^{(n)})}{h_j^{(s, r^{(n)} ) }} \sqrt{w_s^{(r^{(n)})} (x_1^{(n)}) w_s^{(r^{(n)})} (x_2^{(n)})}   \\  & = \underbrace{ \frac{1}{2n^2} \sum_{j = 0}^{\lfloor n \varepsilon \rfloor - 1} \cdots}_{: = \mathscr{S}_n^{(1)} (\varepsilon, z_1, z_2) }  + \underbrace{\frac{1}{2n^2} \sum_{j = \lfloor n \varepsilon \rfloor }^{n-1} \cdots}_{: = \mathscr{S}_n^{(2)} (\varepsilon, z_1, z_2) } .\end{align*} From the previous propositions and Theorem \ref{thm-case2-1}, we know that the following limits all exist $$\lim_{n \to \infty}  \mathscr{S}_n^{(1)} (\varepsilon, z_1, z_2), \quad \lim_{n \to \infty}  \mathscr{S}_n^{(2)} (\varepsilon, z_1, z_2), \quad \lim_{n \to \infty} \Pi_n^{(n)} (z_1, z_2) . $$ By denoting $$ \mathscr{S}^{(1)}_\infty (\varepsilon, z_1, z_2) =  \lim_{n \to \infty}  \mathscr{S}_n^{(1)} (\varepsilon, z_1, z_2) ,   \mathscr{S}^{(2)}_\infty (\varepsilon, z_1, z_2) =  \lim_{n \to \infty}  \mathscr{S}_n^{(2)} (\varepsilon, z_1, z_2) ,$$  we have for any $\varepsilon > 0$, \begin{align}\label{decompose} \Pi_\infty^{(s, v)} (z_1, z_2) =  \mathscr{S}^{(1)}_\infty (\varepsilon, z_1, z_2) + \mathscr{S}^{(2)}_\infty (\varepsilon, z_1, z_2 )  . \end{align} If $z_1 = z_2$, then every term is positive, hence \begin{align*} & \Pi_\infty^{(s, v)} (z_1, z_1)\ge  \mathscr{S}^{(2)}_\infty (\varepsilon, z_1, z_1)  =  \int_\varepsilon^{1} \mathscr{I}^{(s, v)} (\kappa, z_1, z_1) d\kappa. \end{align*} By Cauchy-Schwarz inequality, we can show that $$ | \mathscr{I}^{(s, v)} (\kappa, z_1, z_2)|^2 \le \mathscr{I}^{(s, v)} (\kappa, z_1, z_1) \cdot \mathscr{I}^{(s, v)} (\kappa, z_2, z_2).  $$ Again by Cauchy-Schwarz inequality, we see that $\kappa \to \mathscr{I}^{(s, v)} (\kappa, z_1, z_2)$ is integrable on $(0, 1)$. Combining this fact with \eqref{decompose}, we see that the  limit $ \lim_{\varepsilon \to 0+} \mathscr{S}^{(1)}_\infty (\varepsilon, z_1, z_2) $ always exists. Let us denote  this limit by $ \mathscr{S}_\infty^{(1)} (0, z_1, z_2)$.

Now we show that $\mathscr{S}_\infty^{(1)} (0, z_1, z_2) $ is not identically zero. Let $z_1 = z_2$, then for any $\varepsilon > 0$, we have \begin{align*} &  \mathscr{S}^{(1)}_n (\varepsilon, z_1, z_1) \ge  \frac{1}{2n^2} \frac{\big\{ q_0(x_1^{(n)}) \big\} ^2}{h_0^{(s, r^{(n)} ) }} w_s^{(r^{(n)})} (x_1^{(n)})  \\   = &    \frac{1}{2n^2} \frac{1}{\int_{-1}^{1} w_s^{(r^{(n)})} (t) dt } \frac{ ( 1 - x_1^{(n)} )^s }{ 1 + r^{(n)} - x_1^{(n)}}  \\  = &   \frac{ z_1^s}{v + z_1 }  \frac{1}{ (2n^2)^s \int_{-1}^1 w_s^{(r^{(n)}) } (x) dx }. \end{align*}  We have \begin{align*}  & (2n^2)^s \int_{-1}^1 w_s^{(r^{(n)}) } (x) dx  = \int_0^{4n^2}  \frac{t^s}{  v + t} d t  \\ & \xrightarrow{n \to \infty} \int_0^\infty   \frac{t^s}{ v + t} d t  = v^s \Gamma(-s) \Gamma(s + 1).  \end{align*}   Hence $$ \mathscr{S}_\infty^{(1)}(0, z_1, z_1) \ge  \frac{1}{ v^s \Gamma(-s) \Gamma(s +1) }\cdot \frac{z_1^s}{ v + z_1} \ne 0.  $$

\begin{defn} For $-1< s < 0$,  define a positive function on $\R^*_{+}$:  $$\mathscr{N}^{(s, v)} (z) : = \frac{1}{ (v^s \Gamma(-s) \Gamma(s +1) )^{1/2}}\sqrt{\frac{z^s}{ v + z}}. $$
\end{defn}

\begin{prop}\label{one-rank-per}
For $m = 1$ and $ -1 < s < 0$, we have $$\Pi_\infty^{(s, v)} (z_1, z_2)  = \mathscr{N}^{(s,v)}(z_1) \mathscr{N}^{(s,v)}(z_2) + \int_0^1 \mathscr{I}^{(s, v)} (\kappa, z_1, z_2)   d \kappa .$$\end{prop}

\begin{proof}
By \eqref{decompose}, it suffices to show that $$\mathscr{S}_\infty^{(1)} (0, z_1, z_2)   =  \mathscr{N}^{(s,v)}(z_1) \mathscr{N}^{(s,v)}(z_2). $$ By similar arguments in the proof of Theorem \ref{thm-case2-1}, $ \mathscr{S}_\infty^{(1)} (\varepsilon, z_1, z_2) $ is given by the formula \begin{align}\label{partial-sum}  \mathscr{S}_\infty^{(1)} (\varepsilon, z_1, z_2)  =   \varepsilon \cdot \frac{A_{II}^{(s,v)} (\varepsilon, z_1) B_{II}^{(s,v)} (\varepsilon, z_2)  - A_{II}^{(s,v)} (\varepsilon, z_2) B_{II}^{(s,v)} (\varepsilon, z_1) }{ 2\sqrt{(v + z_1) (v+z_2)} \big[C_{II}^{(s, v)} (\varepsilon) \big]^2 (z_1- z_2) } . \end{align} For $m = 1$,  we have  $$A_{II}^{(s,v)} (\varepsilon, z_i) =  \left |\begin{array}{cc} K_s(\varepsilon \sqrt{v}) & \sqrt{v}  K_s'( \varepsilon \sqrt{v})  \vspace{3mm} \\ J_s(\varepsilon \sqrt{z_i}) & \sqrt{z_i}  J_s' ( \varepsilon \sqrt{z_i})   \end{array} \right|,$$  and  $B_{II}^{(s, v)} (\varepsilon, z_i) = \frac{\partial}{\partial \varepsilon} A_{II}^{(s,v)} (\varepsilon, z_i)$,  $C_{II}^{(s,v)}(\varepsilon) = K_s(\varepsilon \sqrt{v})$. By the differential formula $(\frac{f}{g})' = \frac{f' g - fg' }{g^2} = \frac{1}{g^2} \left| \begin{array}{cc} g & g' \\ f  & f' \end{array} \right|$,  we have  \begin{align*}  A_{II}^{(s,v)}(\varepsilon, z_i ) =  [K_s(\varepsilon \sqrt{v})]^2 \frac{\partial }{ \partial \varepsilon } \left( \frac{J_s(\varepsilon \sqrt{z_i} )}{ K_s(\varepsilon \sqrt{v}) }\right), \end{align*} and \begin{align*} & A_{II}^{(s,v)} (\varepsilon, z_1) B_{II}^{(s,v)} (\varepsilon, z_2)  - A_{II}^{(s,v)} (\varepsilon, z_2) B_{II}^{(s,v)} (\varepsilon, z_1)  \\ = &  [A_{II}^{(s,v)} (\varepsilon, z_1)]^2 \frac{\partial }{\partial \varepsilon} \left(\frac{A_{II}^{(s,v)}(\varepsilon, z_2)}{A_{II}^{(s,v)} (\varepsilon, z_1) }\right).    \end{align*} Hence  \begin{align}\label{iterate-diff}  \begin{split}  & \frac{A_{II}^{(s,v)} (\varepsilon, z_1) B_{II}^{(s,v)} (\varepsilon, z_2)  - A_{II}^{(s,v)} (\varepsilon, z_2) B_{II}^{(s,v)} (\varepsilon, z_1) }{  \big[K_s (\varepsilon \sqrt{v}) \big]^2 }  \\  & = \left(K_s(\varepsilon \sqrt{v})  \right)^2 \left\{    \frac{\partial }{ \partial \varepsilon } \left( \frac{J_s(\varepsilon \sqrt{z_1} )}{ K_s(\varepsilon \sqrt{v}) }\right)  \right\}^2 \frac{\partial }{ \partial \varepsilon }  \left\{ \frac{ \frac{\partial }{ \partial \varepsilon } \left( \frac{J_s(\varepsilon \sqrt{z_2} )}{ K_s(\varepsilon \sqrt{v}) }\right)}{ \frac{\partial }{ \partial \varepsilon } \left( \frac{J_s(\varepsilon \sqrt{z_1} )}{ K_s(\varepsilon \sqrt{v}) }\right) }  \right\}. \end{split} \end{align}  As $\varepsilon\to 0+$, we have \begin{align} \label{iterate1}\left(K_s(\varepsilon \sqrt{v}) \right)^2 \sim \left(\frac{\pi}{2 \sin (s\pi) } \frac{ (\frac{\sqrt{v}}{2})^s }{\Gamma(s + 1)}\right)^2 \varepsilon^{2s}, \end{align} \begin{align} \label{iterate2} \left\{    \frac{\partial }{ \partial \varepsilon } \left( \frac{J_s(\varepsilon \sqrt{z_1} )}{ K_s(\varepsilon \sqrt{v}) }\right)  \right\}^2 \sim \left(\frac{2 \Gamma(s+1) (\frac{\sqrt{z_1}}{2})^s}{\frac{\pi}{ 2 \sin (s \pi) } \Gamma(-s)  (\frac{\sqrt{v}}{2})^{3s}}\right)^2 \varepsilon^{-4s -2} , \end{align}  \begin{align} \label{iterate3} \frac{\partial}{\partial \varepsilon} \left\{ \frac{ \frac{\partial }{ \partial \varepsilon } \left( \frac{J_s(\varepsilon \sqrt{z_2} )}{ K_s(\varepsilon \sqrt{v}) }\right)}{ \frac{\partial }{ \partial \varepsilon } \left( \frac{J_s(\varepsilon \sqrt{z_1} )}{ K_s(\varepsilon \sqrt{v}) }\right) } \right\} \sim \left(\frac{\sqrt{z_2}}{\sqrt{z_1}}\right)^{s}    \frac{\Gamma(-s)}{\Gamma(s+1)} \left(\frac{\sqrt{v}}{2}\right)^{2s} \frac{z_1- z_2}{2} \varepsilon^{2s+1}. \end{align}

For example,  let us check the asymptotic formula \eqref{iterate3}. We have $$ \frac{J_s(\varepsilon \sqrt{z_i} )}{ K_s(\varepsilon \sqrt{v}) }   = \frac{2 \sin (s\pi) }{\pi} \left(\frac{\sqrt{z_i}}{2}\right)^s \frac{\mathscr{F}(\varepsilon^2 z_i)}{\mathscr{G}(\varepsilon^2) - \varepsilon^{-2s} \mathscr{H}(\varepsilon^2) }, $$  where $\mathscr{F}, \mathscr{G}, \mathscr{H}$ are entire functions given by $$\mathscr{F}(z)= \sum_{\nu=0}^\infty \frac{(-1)^{\nu} (\frac{z}{4})^\nu}{\nu ! \Gamma(\nu + s + 1)}, \quad \mathscr{G}(z)= \left(\frac{\sqrt{v}}{2}\right)^s\sum_{\nu=0}^\infty \frac{ (\frac{z}{4})^\nu}{\nu ! \Gamma(\nu + s + 1)} , $$ $$   \mathscr{H}(z)= \left(\frac{\sqrt{v}}{2}\right)^{-s}\sum_{\nu=0}^\infty \frac{ (\frac{z}{4})^\nu}{\nu ! \Gamma(\nu - s + 1)} . $$ It follows that \begin{align*}   \frac{\partial}{\partial \varepsilon}\left\{ \frac{ \frac{\partial }{ \partial \varepsilon } \left( \frac{J_s(\varepsilon \sqrt{z_2} )}{ K_s(\varepsilon \sqrt{v}) }\right)}{ \frac{\partial }{ \partial \varepsilon } \left( \frac{J_s(\varepsilon \sqrt{z_1} )}{ K_s(\varepsilon \sqrt{v}) }\right) }\right\} = \left(\frac{\sqrt{z_2}}{\sqrt{z_1}}\right)^{s}  2\varepsilon \cdot  \underbrace{ \frac{\partial }{\partial x } \left[ \frac{        \frac{\partial }{\partial x} \left( \frac{\mathscr{F}(xz_2)}{ \mathscr{G} (x) - x^{-s} \mathscr{H}(x)}\right)             }{      \frac{\partial }{\partial x} \left( \frac{\mathscr{F}(xz_1)}{ \mathscr{G} (x) - x^{-s} \mathscr{H}(x)}\right)        }\right](\varepsilon^2)}_{= : \,  Q(\varepsilon^2)} .  \end{align*}  For $i = 1, 2$,  let us denote \begin{align*} Q_i(x) = &  z_i\mathscr{F}'(xz_i) \Big(x^{s+1}\mathscr{G}(x)   - x \mathscr{H}(x)\Big)  \\ & -  \mathscr{F}(xz_i) \Big( x^{s+1} \mathscr{G}'(x) + s  \mathscr{H}(x) - x \mathscr{H}'(x) \Big) , \end{align*} then $Q(x) = \frac{\partial}{ \partial x } \left[ \frac{ Q_2(x) }{Q_1(x) }\right].$  Note that $Q_1(0) = Q_2(0) = -s \mathscr{F}(0) \mathscr{H}(0)$. Now we  obtain that, as $x \to 0 + $, \begin{align*} Q (x) \sim  & \frac{(s+1) x^s}{Q_1(0)^2}  \cdot \Big\{ Q_2(0) \big( z_2 \mathscr{F}'(0) \mathscr{G}(0) - \mathscr{F}(0) \mathscr{G}'(0)\big)  \\ & -  Q_1(0) \big( z_1 \mathscr{F}'(0) \mathscr{G}(0) - \mathscr{F}(0) \mathscr{G}'(0)\big) \Big\} \\ \sim &  \frac{(s+1) x^s}{-s \mathscr{F}(0) \mathscr{H}(0)} \mathscr{F}'(0) \mathscr{G}(0) (z_2-z_1)  \\ \sim &  \frac{\Gamma(-s)}{\Gamma(s+1)} \left(\frac{\sqrt{v}}{2}\right)^{2s} \frac{z_1- z_2}{4} x^s.\end{align*} Combining the above asymptotics, we get \eqref{iterate3}.

Substituting \eqref{iterate1}, \eqref{iterate2} and \eqref{iterate3} to \eqref{iterate-diff}, we have \begin{align*}  & \frac{A_{II}^{(s,v)} (\varepsilon, z_1) B_{II}^{(s,v)} (\varepsilon, z_2)  - A_{II}^{(s,v)} (\varepsilon, z_2) B_{II}^{(s,v)} (\varepsilon, z_1) }{  \big[K_s (\varepsilon \sqrt{v}) \big]^2 }   \\ \sim & \frac{2(z_1 -z_2)}{\Gamma(-s) \Gamma(s+1)} \left(\frac{\sqrt{z_1z_2}}{ v}\right)^s \varepsilon^{-1}, \text{ as } \varepsilon \to 0+.   \end{align*} Finally, by \eqref{partial-sum}, we get the formula for $ \mathscr{S}_\infty^{(1)} (0, z_1, z_2) $:    \begin{align*}      \frac{1}{\Gamma(-s) \Gamma(s  + 1)} \frac{1}{ \sqrt{(v+z_1) (v+z_2)}} \left(\frac{\sqrt{z_1z_2}}{v}\right)^s  = \mathscr{N}^{(s,v)}(z_1) \mathscr{N}^{(s,v)}(z_2). \end{align*}

\end{proof}

For $\alpha > -1$, we denote by $\widetilde{J}^{(\alpha)} (x, y)$ the Bessel kernel, i.e., $$\widetilde{J}^{(\alpha)}(x, y) = \frac{J_\alpha(\sqrt{x}) \sqrt{y} J_\alpha'(\sqrt{y}) - J_\alpha(\sqrt{y}) \sqrt{x} J_\alpha'(\sqrt{x})}{2(x-y)} .$$ It is well-known (cf. e.g. \cite{Tracy-Widom94}) that the Bessel kernel has the following integral representation: $$\widetilde{J}^{(\alpha)} (x, y) =\frac{1}{4} \int_0^1 J_\alpha(\sqrt{t x}) J_\alpha(\sqrt{ty}) dt .$$

\begin{prop}
Let $m = 1$ and $ - 1 < s < 0$. Then $$\lim_{ v \to 0+} \Pi_\infty^{(s,v)} (z_1, z_2) =\widetilde{J}^{(s +1)} (z_1, z_2).$$ Moreover, the convergence is uniform as soon as $z_1, z_2$ are in compact subsets of $(0, \infty)$.
\end{prop}

\begin{proof}
Fix $-1< s < 0$. It is easy to see that  $$\lim_{v \to 0+} \mathscr{N}^{(s,v)} (z) = 0, $$ and the convergence is  uniform for $ z$ in compact subset of $ (0, \infty)$.

By \eqref{differential-relation-J} and \eqref{differential-relation-K}, we have $$  A_{II}^{(s,v)}(\kappa, z_i)   =  \left |\begin{array}{cc} K_s(\kappa \sqrt{v}) & -   \sqrt{v}  K_{s+1}( \kappa \sqrt{v})  \vspace{3mm} \\ J_s(\kappa \sqrt{z_i}) &  - \sqrt{z_i}J_{s+1} ( \kappa \sqrt{z_i})   \end{array} \right|.  $$ Then apply the asymptotics of $K_s$, $K_{s +1}$ near 0 to get \begin{align*} & \lim_{v \to \infty} \frac{A_{II}^{(s,v)} (\kappa, z_i)}{K_s(\kappa \sqrt{v})}   = - \sqrt{z_i} J_{s + 1} ( \kappa \sqrt{z_i}).\end{align*}   It follows that $$\lim_{v \to 0 + } \mathscr{I}^{(s,v)} (\kappa, z_1, z_2)  = \frac{J_{s + 1}(\kappa \sqrt{z_1}) J_{ s + 1} ( \kappa \sqrt{z_2})}{2 } \cdot \kappa.$$ For any $0 < \varepsilon < 1$,   the convergence is uniform as long as $\kappa\in [\varepsilon, 1]$ and $z_1, z_2$ in compact subsets of $(0, \infty)$. Hence $$\lim_{v \to 0+ } \int_\varepsilon^1 \mathscr{I}^{(s,v)} (\kappa, z_1, z_2)  d \kappa =    \int_{\varepsilon}^1 \frac{J_{s + 1}(\kappa \sqrt{z_1}) J_{ s + 1} ( \kappa \sqrt{z_2})}{2} \cdot \kappa d \kappa. $$ The above term tends to $$\int_0^1 \frac{J_{s + 1}(\kappa \sqrt{z_1}) J_{ s + 1} ( \kappa \sqrt{z_2})}{2} \cdot \kappa d \kappa  =   \frac{1}{4 } \int_0^1  J_{s + 1}(\sqrt{t z_1}) J_{ s + 1} (\sqrt{tz_2} )  dt $$ uniformly as $z_1, z_2$ in compact subsets of $(0, \infty)$, when $\varepsilon \to 0$.  It is easy to see that $$\sup_{0< v< R,\\ r < |z_1|, |z_2| < R  }\left| \int_0^\varepsilon \mathscr{I}^{(s,v)} (\kappa, z_1, z_2) d\kappa \right| \lesssim \varepsilon^{s +1}.$$ Hence $$\lim_{ v \to 0+} \Pi_\infty^{(s,v)} (z_1, z_2) =\widetilde{J}^{(s +1)} (z_1, z_2),$$ with uniform convergence as long as $z_1, z_2$ are in compact subsets of $(0, \infty)$.

\end{proof}

\begin{rem}
When $m \ge 2$ and $ -1< s < m-1$,  the situation is similar, but the formula and the proof will be slightly tedious.
\end{rem}

\subsection{Explicit Kernels for Scaling Limit: Case III}  Let $s> -1$. We  consider in this section  a sequence of positive real numbers $\gamma^{(n)}$ and modify the Jacobi weights given by $$\widehat{w}_{s}^{(n)} (t)  = \frac{w_{s}(t)}{( 1 + \gamma^{(n)} - t)^2} = \frac{(1 - t)^{s}}{(1 + \gamma^{(n)}- t)^2}. $$ The $n$-th Christoffel-Darboux kernel associated with $\widehat{w}_{ s}^{(n)}$ is denoted by $\Phi_n^{(n)} (x_1, x_2)$. We will investigate the scaling limit of $\Phi_n^{(n)} ( x_1^{(n)}, x_2^{(n)})$ in the regime: \begin{align}\label{case3} \begin{split}  x_i^{(n)} = 1  - \frac{z_i}{2n^2}   \, & \text{ with } z_i > 0 , i = 1, 2. \\  \gamma^{(n)} = \frac{u}{2n^2} \, &  \text{ with }  u > 0.  \end{split} \end{align}

\subsubsection{Explicit formulae for orthogonal polynomials and Christoffel-Darboux kernels}

For $j \ge 2$,  we set  $$p_j^{(n)} (t): = \left|\begin{array}{ccc}Q_{j-2}^{(s)} (1 + \gamma^{(n)}) & Q_{j-1}^{(s)} (1 + \gamma^{(n)}) & Q_{j}^{(s)} (1 + \gamma^{(n)}) \vspace{3mm}  \\R_{j-2}^{(s)} (1 + \gamma^{(n)}) & R_{j-1}^{(s)} (1 + \gamma^{(n)}) & R_{j}^{(s)} (1 + \gamma^{(n)}) \vspace{3mm}\\P_{j-2}^{(s)} (t) & P_{j-1}^{(s)} (t) & P_{j}^{(s)} (t)\end{array}\right|; $$ $$ e_{j}^{(n)} : = \left|\begin{array}{cc}Q_{j-2}^{(s)} (1 + \gamma^{(n)}) & Q_{j-1}^{(s)} (1 + \gamma^{(n)}) \vspace{3mm} \\R_{j-2}^{(s)} (1 + \gamma^{(n)}) & R_{j-1}^{(s)} (1 + \gamma^{(n)})\end{array}\right|.$$ The leading term of $p_j^{(n)}$ is $\widehat{k}_j^{(n)} = e_j^{(n)} k_j^{(s)}.$

\begin{prop}
For $j \ge 2$, the polynomial $q_j^{(n )}$ is the $j$-th orthogonal polynomial with respect to the weight $\widehat{w}^{(n)}_s$ on $ [-1, 1]$.
\end{prop}

\begin{proof}
By the Uvarov formula, we know that for $ j \ge 1$, $$\widehat{p}_j^{(n)} (t) = \left| \begin{array}{cc}  Q_{j-1}^{(s)} (1 + \gamma^{(n)} ) & Q_{j}^{(s)} (1 + \gamma^{(n)} )  \vspace{3mm} \\ P_{j-1}^{(s)} (t ) & P_{j}^{(s)} (t )  \end{array} \right|$$ is the $j$-th orthogonal polynomial with respect to the weight $\frac{w_s(t) }{ 1 + \gamma^{(n)} - t }$. Applying the Uvarov formula again, we know that for $j \ge 2$,  \begin{align}\label{iterated-uvarov} \left| \begin{array}{cc} \mathscr{C}_{ j - 1}(1 + \gamma^{(n)}) & \mathscr{C}_j(1 + \gamma^{(n)})   \vspace{2mm} \\ \widehat{p}_{j-1}^{(n)} (t ) & \widehat{p}_j^{(n)} (t) \end{array}\right| \end{align} is the $j$-th orthogonal polynomial with respect to the weight $\frac{w_s(t) }{( 1 + \gamma^{(n)} - t)^2}$, where we denote  by $$\mathscr{C}_j(x) =\int_{-1}^1 \frac{\widehat{p}_j^{(n)} (t) }{x - t} \cdot \frac{w_s(t)}{ 1 + \gamma^{(n)} - t}dt . $$ We can easily verify  that the polynomial \eqref{iterated-uvarov} is a multiple of $p_j^{(n)}$. \end{proof}

\begin{defn}
For  $j \ge 2 $,  denote $$ \widehat{h}_{ j}^{(n) } = \int_{-1}^{1} \big[ p_j^{(n) } (t) \big]^2   \frac{(1 - t)^{s}}{(1  + \gamma^{(n)} - t)^2} dt. $$
\end{defn}

\begin{prop}
For $ j \ge 2$, we have the identity \begin{eqnarray*}  \widehat{h}_{j}^{(n) }  =  \frac{e_{ j}^{(n)}   e_{ j + 1}^{(n)}  k_{j }^{(s)}  h_{j-2}^{(s )}}{k_{j-2}^{(s)}} .\end{eqnarray*}
\end{prop}

\begin{proof}
By the orthogonality property,  we have \begin{align*}  \widehat{h}_j^{(n) } & =  \int_{-1}^1 p_j^{(n)}(t)   e_{ j}^{(n)} P_{j }^{(s)} (t) \frac{(1 - t)^s}{(1  + \gamma^{(n)} - t)^2} dt \\ & =  e_{ j}^{(n)}   k_{j }^{(s)}  \int_{-1}^1 p_j^{(n)}(t)   \cdot  t^j \cdot  \frac{(1 - t)^{s}}{(1  + \gamma^{(n)} - t)^2} dt \\ & =  e_{j}^{(n)}   k_{j }^{(s)}  \int_{-1}^1 p_j^{(n)}(t)   \cdot  t^{j-2}  \cdot ( 1 + \gamma^{(n)} - t)^2 \cdot  \frac{(1 - t)^{s}}{(1  + \gamma^{(n)} - t)^2} dt  \\ & =   e_{ j}^{(n)}   k_{j }^{(s)}  \int_{-1}^1 p_j^{(n)}(t)   \cdot  t^{j-2} \cdot  w_{s} (t) dt \\ & =  e_{j}^{(n)}   k_{j }^{(s )}  \int_{-1}^1   e_{j + 1}^{(n)}   P_{j-2}^{(s)}(t)  \cdot  t^{j-2} \cdot w_{s } (t) dt \\ & =  \frac{e_{j}^{(n)}   e_{j + 1}^{(n)}  k_{j }^{(s )} }{k_{j-2}^{(s )}} \int_{-1}^1   \Big[ P_{j-2}^{(s )}(t) \Big]^2 w_{s } (t) dt  \\ & =  \frac{e_{j}^{(n)}   e_{ j + 1}^{(n)}  k_{j }^{(s)} }{k_{j-2}^{(s)}} h_{j-2}^{(s )}. \end{align*}
\end{proof}

The Christoffel-Darboux kernels $\Phi_n^{(n)}$ in the $(x_1^{(n)}, x_2^{(n)})$-coodinates are given by  \begin{align*}  & \Phi_n^{(n)} ( x_1^{(n)}, x_2^{(n)})  =  \sqrt{ \widehat{w}_{s}^{(n)}  (x_1^{(n)})   \widehat{w}_{s}^{(n)} (x_2^{(n)}) }  \sum_{\ell=0}^{n-1} \frac{p_\ell^{(n)} (x_1^{(n)})  p_\ell^{(n)} (x_2^{(n)}) }{\widehat{h}_\ell^{(n)}} \\ = &  \frac{ \sqrt{ \widehat{w}_{s}^{(n)}  (x_1^{(n)})   \widehat{w}_{s}^{(n)} (x_2^{(n)}) }    }{ \frac{\widehat{h}_{n-1}^{(n)} \widehat{k}_n^{(n)}}{ \widehat{k}_{n - 1}^{(n)}}   }  \cdot  \frac{p_n^{(n)} (x_1^{(n)})  p_{n-1}^{(n)} (x_2^{(n)}) - p_n^{(n)} (x_2^{(n)})  p_{n - 1}^{(n)} (x_1^{(n)}) }{ x_1^{(n)} - x_2^{(n)} }.  \end{align*} These kernels in the $(z_1, z_2)$-coodinates are given by $$ \widetilde{\Phi}_n^{(n)}( z_1, z_2) =  \frac{1}{2n^2}\Phi_n^{(n)} \Big( 1 - \frac{z_1}{2n^2}, 1 - \frac{z_2}{2n^2}\Big).$$

\subsubsection{Scaling limits}

\begin{prop}\label{matrix-AC-3}
In the regime \eqref{case3}, we have $$\lim_{n \to \infty} n^{-1-2s}   e_{\kappa_n}^{(n)} = 2^{2s+\frac{3}{2}} u^{-1-s} C_{ III}^{(s ,u )}(\kappa); $$  $$\lim_{n \to\infty} n^{1-3s} p_{\kappa_n}^{(n )} ( x_i^{(n)})  = 2^{3s+\frac{3}{2}} u^{-1-s} z^{-\frac{s}{2}} A_{ III}^{(s , u)}(\kappa, z_i),$$ where $$C_{ III}^{(s , u)}(\kappa)  = \left| \begin{array}{cc}K_{s} (\kappa \sqrt{u}) & u^{\frac{1}{2}}K_{s}'(\kappa \sqrt{u}) \vspace{3mm} \\ L_{s}(\kappa \sqrt{u}) & u^{\frac{1}{2}} L_{s}'(\kappa \sqrt{u})\end{array}\right| $$ and $$ A_{ III}^{(s, u)}(\kappa, z) = \left| \begin{array}{ccc}K_{s}(\kappa \sqrt{u}) & u^{\frac{1}{2}}K_{s}'(\kappa \sqrt{u}) & u K_{s}''(\kappa \sqrt{u}) \vspace{3mm} \\ L_{s}(\kappa \sqrt{u}) & u^{\frac{1}{2}} L_{s}'(\kappa \sqrt{u}) & u L_{s}''(\kappa \sqrt{u}) \vspace{3mm} \\J_{s}(\kappa \sqrt{z}) & z^{\frac{1}{2}}J_{s}'(\kappa \sqrt{z}) & z J_{s}''(\kappa \sqrt{z})\end{array}\right|. $$ Moreover, for any $\varepsilon > 0$,  the convergences are uniform as long as $ \kappa \in [\varepsilon, 1 ]$ and $z_i$ ranges compact simple connected subset of $\C\setminus \{0\}$.
\end{prop}

\begin{proof}
The proof is similar to that of Proposition \ref{C}.
\end{proof}

\begin{prop}
In the regime  \eqref{case3}, we have $$\lim_{n \to \infty} n^{2-3s} \Big\{p_{\kappa_n}^{(n)} (x_i^{(n)})   - p_{\kappa_n-1}^{(n)} (x_i^{(n)})   \Big\} = 2^{3s + \frac{3}{2}} u^{-1-s} z_i^{-\frac{s}{2}} B_{ III}^{(s, u)} (\kappa, z_i) $$ and \begin{align*} &  \lim_{n \to \infty} n^{3- 6s } \Big\{ p_{\kappa_n}^{(n)} (x_1^{(n)})  p_{\kappa_n-1}^{(n)} (x_2^{(n)}) - p_{\kappa_n}^{(n)} (x_2^{(n)})  p_{\kappa_n - 1}^{(n)} (x_1^{(n)})  \Big\} \\   = &   2^{6s + 3} u^{-2 - 2s} (z_1z_2)^{-\frac{ s }{2}}  \Big\{  B_{ III}^{(s, u)} ( \kappa, z_1)  A_{ III}^{(s , u )} ( \kappa, z_2)   -  B_{ III}^{(s, u )} ( \kappa, z_2)  A_{ III}^{(s, u )} ( \kappa, z_1)  \Big\} , \end{align*} where $B_{ III}^{(s, u )} (\kappa , z) = \frac{\partial}{\partial \kappa} A_{ III}^{(s, u)}(\kappa, z)$, i.e., $$   B_{ III}^{(s, u )} (\kappa , z)  = \left | \begin{array}{ccc}K_{s}(\kappa \sqrt{u}) & u^{\frac{1}{2}}K_{s}'( \kappa \sqrt{u}) & u^{\frac{3}{2} } K_{s}^{(3)}(\kappa  \sqrt{u}) \vspace{3mm} \\ L_{s}( \kappa \sqrt{u}) & u^{\frac{1}{2}} L_{s}'( \kappa \sqrt{u}) & u^{\frac{3}{2}} L_{s}^{(3)}(\kappa \sqrt{u})  \vspace{3mm}  \\J_{s} (  \kappa \sqrt{z}) & z^{\frac{1}{2}}J_{s}'(  \kappa \sqrt{z}) & z^{\frac{3}{2}} J_{s}^{(3)} (\kappa \sqrt{z})\end{array}\right | .$$  Moreover, for any $\varepsilon > 0$,  the convergences are uniform as long as $ \kappa \in [\varepsilon, 1 ]$ and $z_i$ ranges compact simple connected subset of $\C\setminus \{0\}$.
\end{prop}

\begin{proof}
The proof is similar to that of Proposition \ref{B}.
\end{proof}

Now we obtain the following theorem.
\begin{thm}
 In the regime \eqref{case3}, we obtain the scaling limit \begin{align*} &  \Phi_\infty^{(s, u)}  (z_1, z_2) : = \lim_{n \to \infty} \widetilde{\Phi}_n^{(n)} ( z_1, z_2) \\ = & \frac{  A_{ III}^{(s, u)} ( 1, z_1)  B_{ III}^{(s , u )} ( 1, z_2)   -  A_{ III}^{(s, u )} ( 1, z_2)  B_{ III}^{(s, u )} ( 1, z_1) }{2 ( u + z_1)(u+z_2)  \cdot \big[ C_3^{(s,u)} (1) \big]^2 \cdot (z_1- z_2) } . \end{align*}
\end{thm}

For investigating the integral form of the scaling limit $\Phi_\infty^{(s,u)}$, let us  put $$p_0^{(n)} (t) \equiv 1 \text{ and } \, p_1^{(n)} (t) =1- t - \frac{1}{\widehat{h}_0^{(n)}} \int_{-1}^1(1 -t) \widehat{w}_s^{(n)}(t) d t .$$  The contribution of $p_0^{(n)}$ to the kernel is  \begin{align*} &   \frac{\sqrt{\widehat{w}_s^{(n)}(x_1^{(n)})  \widehat{w}_s^{(n)}(x_2^{(n)}) }}{2n^2} \cdot \frac{p_0^{(n)} (x_1^{(n)}) p_0^{(n)} (x_2^{(n)}) }{\widehat{h}_0^{(n)}}  = \frac{ (z_1z_2)^{\frac{s}{2}}}{(z_1 + u)(z_2+ u)} \cdot \frac{(2n^2)^{1-s}}{\int_{-1}^1 \frac{ (1-t)^s }{ ( 1 + \frac{u}{2n^2} - t)^2} dt }.\end{align*} We note that for $ -1< s < 1$, we have  \begin{align*} & \frac{\int_{-1}^1 \frac{ (1-t)^s }{ ( 1 + \frac{u}{2n^2} - t)^2} dt }{ (2n^2)^{1-s}} = \int_0^{4n^2} \frac{t^s}{ (t + u)^2} d t\\ \xrightarrow{n \to \infty} &  \int_{0}^{\infty} \frac{t^s}{(t + u)^2} dt  =  u^{s-1} B(1 + s, 1-s) = u^{s-1} \Gamma(1  + s) \Gamma(1-s).\end{align*} The contribution of $p_1^{(n)}$ to the kernel is    \begin{align*}  &   \frac{\sqrt{\widehat{w}_s^{(n)}(x_1^{(n)})  \widehat{w}_s^{(n)}(x_2^{(n)}) }}{2n^2} \cdot \frac{p_1^{(n)} (x_1^{(n)}) p_1^{(n)} (x_2^{(n)}) }{\widehat{h}_1^{(n)}}  \\ = &  \frac{ (z_1z_2)^{\frac{s}{2}}}{(z_1 + u)(z_2+ u)}  \cdot \frac{(2n^2)^{1-s} p_1^{(n)} (x_1^{(n)}) p_1^{(n)} (x_2^{(n)})  }{\widehat{h}_1^{(n)}}. \end{align*}  For $ - 1 < s < 0$, we have \begin{align*}    \frac{1}{\widehat{h}_0^{(n)}} \int_{-1}^1(1 -t) \widehat{w}_s^{(n)}(t) d t  = \frac{ \int_{0}^{4n^2} \frac{ t^{s + 1}}{ ( t + u)^2} dt}{2n^2  \int_{0}^{4n^2} \frac{ t^{s }}{ ( t + u)^2} dt }. \end{align*} Hence \begin{align*} p_1^{(n)} ( x_i^{(n)} ) = \frac{1}{2n^2} \left( z_i - \frac{ \int_{0}^{4n^2} \frac{ t^{s + 1}}{ ( t + u)^2} dt}{  \int_{0}^{4n^2} \frac{ t^{s }}{ ( t + u)^2} dt } \right),\end{align*}
 and \begin{align*} \widehat{h}_1^{(n)} =  (2n^2)^{-1-s} \int_{0}^{4n^2} \left( y - \frac{ \int_{0}^{4n^2} \frac{ t^{s + 1}}{ ( t+ u)^2} dt}{  \int_{0}^{4n^2} \frac{ t^{s }}{ ( t + u)^2} dt }\right)^2 dy .\end{align*} It follows that \begin{align*} & \lim_{n \to \infty} \frac{(2n^2)^{1-s} p_1^{(n)} (x_1^{(n)}) p_1^{(n)} (x_2^{(n)})  }{\widehat{h}_1^{(n)}}    =  \frac{ \left( z_1- \frac{ \int_{0}^{\infty} \frac{ t^{s + 1}}{ ( t + u)^2} dt}{  \int_{0}^{\infty} \frac{ t^{s }}{ ( t + u)^2} dt } \right) \left( z_2 - \frac{ \int_{0}^{\infty} \frac{ t^{s + 1}}{ ( t + u)^2} dt}{  \int_{0}^{\infty} \frac{ t^{s }}{ ( t + u)^2} dt } \right)}{ \int_{0}^{\infty} \left( y - \frac{ \int_{0}^{\infty} \frac{ t^{s + 1}}{ ( t+ u)^2} dt}{  \int_{0}^{\infty} \frac{ t^{s }}{ ( t + u)^2} dt }\right)^2    \frac{y^s}{(y + u)^2} dy } \\ & =   \frac{ \left( z_1 + \frac{1+s}{s} u \right) \left( z_2  + \frac{1+s}{s} u \right)}{ \int_{0}^{\infty} \left( y + \frac{1+s}{s} u \right)^2  \frac{y^s}{(y + u)^2} dy }  .\end{align*}

\begin{defn}
For $ -1< s < 1$, define a positive function on $\R^{*}_{+}$ :  $$\mathscr{M}^{(s,u)}_0 (z) : = \frac{1}{\sqrt{u^{s-1} \Gamma( 1 + s) \Gamma)( 1 + s ) }} \cdot \frac{z^{\frac{s}{2}}}{z + u }.$$  For $ -1 < s< 0$, define $$ \mathscr{M}_1^{(s,u) } (z) : =  \frac{ 1 }{ \left[  \int_{0}^{\infty} \left( y + \frac{1+s}{s} u \right)^2  \frac{y^s}{(y + u)^2} dy \right]^{1/2}} \cdot (z + \frac{1+s}{s} u ) \cdot\frac{z^{\frac{s}{2}}}{ z + u },  $$ we  extend the definition of $ \mathscr{M}_1^{(s,u) }$ when $ 0 \le s < 1$ by setting $ \mathscr{M}_1^{(s,u) } \equiv 0$.
\end{defn}

The detail proof of the following proposition is long but routine and similar to the proof of Proposition \ref{one-rank-per}, so we omit its proof here.
\begin{prop}
For $ -1< s < 1$, we have the following representation of $\Phi_\infty^{(s, u)} ( z_1, z_2)$:  \begin{align*}   \Phi_\infty^{(s, u)} (z_1, z_2) = &  \mathscr{M}_0^{(s,u)} (z_1) \mathscr{M}_0^{(s,u)} (z_2)   +  \mathscr{M}_1^{(s,u)} (z_1) \mathscr{M}_1^{(s,u)} (z_2)   \\ & + \frac{1}{2 ( z_1 + u )( z_2 + u )}\int_0^1 \frac{A_{ III}^{(s,u)} (\kappa, z_1) A_{ III}^{(s,u)} (\kappa, z_2) }{ \big[ C_{ III}^{(s,u) } (\kappa) \big]^2} \kappa d \kappa.  \end{align*}
\end{prop}

\def\cprime{$'$}



\end{document}